\documentclass[12pt,reqno]{amsart}
\usepackage[margin=1in]{geometry}
\usepackage{amsmath,amssymb,amsthm,graphicx,amsxtra, setspace}
\usepackage[utf8]{inputenc}
\usepackage{mathrsfs}
\usepackage{alltt}
\usepackage{relsize}
\usepackage{hyperref}
\usepackage{aliascnt}
\usepackage{tikz}
\usepackage{mathtools}
\usepackage{multicol}
\usepackage{upgreek}
\usepackage{graphicx,type1cm,eso-pic,color}
\allowdisplaybreaks
\usepackage{scalerel,stackengine}
\stackMath
\newcommand\reallywidehat[1]{%
	\savestack{\tmpbox}{\stretchto{%
			\scaleto{%
				\scalerel*[\widthof{\ensuremath{#1}}]{\kern-.6pt\bigwedge\kern-.6pt}%
				{\rule[-\textheight/2]{1ex}{\textheight}}
			}{\textheight}%
		}{0.5ex}}%
	\stackon[1pt]{#1}{\tmpbox}%
}
\numberwithin{equation}{section}

\newtheorem{theorem}{Theorem}[section]

\newaliascnt{lemma}{theorem}
\newtheorem{lemma}[lemma]{Lemma}
\aliascntresetthe{lemma} 

\newaliascnt{proposition}{theorem}
\newtheorem{proposition}[proposition]{Proposition}
\aliascntresetthe{proposition}

\newaliascnt{assumption}{theorem}

\aliascntresetthe{assumption}

\newaliascnt{corollary}{theorem}

\aliascntresetthe{corollary}

\newaliascnt{definition}{theorem}
\newtheorem{definition}[definition]{Definition}
\aliascntresetthe{definition}

\newaliascnt{example}{theorem}

\aliascntresetthe{example}

\newaliascnt{remark}{theorem}
\newtheorem{remark}[remark]{Remark}
\aliascntresetthe{remark}

\newaliascnt{hypothesis}{theorem}

\aliascntresetthe{hypothesis}

\newaliascnt{property}{theorem}

\aliascntresetthe{property}

\let\originalleft\left
\let\originalright\right
\renewcommand{\left}{\mathopen{}\mathclose\bgroup\originalleft}
\renewcommand{\right}{\aftergroup\egroup\originalright}

\makeatletter
\newcommand{\doublewidetilde}[1]{{%
		\mathpalette\double@widetilde{#1}%
}}
\newcommand{\double@widetilde}[2]{%
	\sbox\z@{$\m@th#1\widetilde{#2}$}%
	\ht\z@=.9\ht\z@
	\widetilde{\box\z@}%
}
\makeatother

\renewcommand{\d}{\/\mathrm{d}\/}

\def\w{\textbf{W}^{\varepsilon}_{{\theta}^{\varepsilon}}}

\def\S{\mathrm{S}}
\def\T{\mathrm{T}}
\def\L{\mathrm{L}}
\def\A{\mathrm{A}}
\def\I{\mathrm{I}}
\def\F{\mathrm{F}}
\def\C{\mathrm{C}}

\def\h{\mathbf{h}}
\def\J{\mathrm{J}}
\def\B{\mathrm{B}}
\def\D{\mathrm{D}}
\def\y{\mathbf{y}}
\def\Y{\mathbb{Y}}
\def\Z{\mathrm{Z}}

\def\X{\mathbb{X}}
\def\x{\mathbf{x}}
\def\z{\mathbf{z}}
\def\v{\mathbf{v}}
\def\V{\mathbb{V}}
\def\w{\mathbf{w}}
\def\W{\mathrm{W}}
\def\G{\mathbb{G}}

\def\M{\mathrm{M}}

\def\no{\nonumber}
\def\V{\mathbb{V}}

\def\u{\mathbf{u}}
\def\H{\mathbb{H}}
\def\n{\mathbf{n}}

\newcommand{\R}{\mathbb{R}}

\renewcommand{\d}{\/\mathrm{d}\/}

\newcommand{\Addresses}{{
		\footnote{
			\emph{Keywords: }{Steady state nonlocal Cahn-Hilliard-Navier-Stokes systems, maximal monotone operator, pseudo-monotonicity, exponential stability}
			
			\emph{2020 Mathematics Subject Classification: }{49J20, 35Q35, 76D03.}

			\noindent \textsuperscript{1}School of Mathematics,
			Indian Institute of Science Education and Research Thiruvananthapuram (IISER-TVM),
			Maruthamala PO, Vithura, Thiruvananthapuram, Kerala, 695 551, INDIA.  \par\nopagebreak \noindent
			\textit{e-mail: }\texttt{tania0003@gmail.com} (Tania Biswas)
			
			\textit{e-mail: }\texttt{sheetal@iisertvm.ac.in} (Sheetal Dharmatti)
			
			\textit{e-mail: }\texttt{plnmn915@iisertvm.ac.in}	(Perisetti Lakshmi Naga Mahendranath)
			
			*Corresponding Author

			\noindent \textsuperscript{2}Department of Mathematics,
			Indian Institute of Technology Roorkee-IIT Roorkee,
			Haridwar Highway, Roorkee,
			Uttarakhand 247667,  \par\nopagebreak \noindent
			\textit{e-mail: }\texttt{maniltmohan@gmail.com, manilfma@iitr.ac.in} (Manil T. Mohan)

}}}
\begin{document}
	
	\title[Stationary Nonlocal Cahn-Hilliard-Navier-Stokes System]{On the Stationary Nonlocal  Cahn-Hilliard-Navier-Stokes System: Existence, Uniqueness and Exponential Stability  \Addresses   }
	
	%
	%
	%
	%
	%
	%
	%
	%
	%
	%
	%
	\author[T. Biswas, S. Dharmatti, L. N. M. Perisetti and M. T. Mohan]
	{Tania Biswas\textsuperscript{1},  
		Sheetal Dharmatti\textsuperscript{1*},  
		Perisetti Lakshmi Naga Mahendranath\textsuperscript{1}
		and Manil T Mohan\textsuperscript{2}}

	\begin{abstract}
		Cahn-Hilliard-Navier-Stokes system describes  the evolution of two isothermal, incompressible, immiscible fluids in a bounded domain.  In this work, we consider the stationary nonlocal Cahn-Hilliard-Navier-Stokes system in two and three dimensions with singular potential. We prove the existence of a weak solution for the system using pseudo-monotonicity arguments and Browder's theorem. Further, we establish the uniqueness and regularity results for the weak solution of the stationary nonlocal Cahn-Hilliard-Navier-Stokes system for constant mobility parameter and viscosity.  Finally, in two dimensions, we establish that the stationary solution  is exponentially stable (for convex singular potentials) under suitable conditions on mobility parameter and viscosity.
	\end{abstract}

	\maketitle
	
	\section{Introduction}\label{sec1}\setcounter{equation}{0} 	
	We  consider a mathematical model   of two isothermal, incompressible, immiscible fluids evolving in two or three dimensional bounded domains. This system of equations is well known as  \emph{Cahn-Hilliard-Navier-Stokes(CHNS) system} or is also known as $\mathrm{H}$-Model. Cahn-Hilliard-Navier-Stokes model describes the chemical interactions between the two phases at the interface, which is achieved using a Cahn-Hilliard approach, and also the hydrodynamic properties of the mixture which is obtained by using Navier-Stokes equations with surface tension terms acting at the interface (see \cite{MR2580516}). If the two fluids have the same constant density, then the temperature differences are negligible and the diffusive interface between the two phases has a small but non-zero thickness, and thus we have the well-known ``H-Model" (see \cite{MR1404829}). The equations for evolution of the Cahn-Hilliard-Navier-Stokes/H-model are given by
	\begin{equation}\label{1.1}
	\left\{
	\begin{aligned}
	\varphi_t + \u\cdot \nabla \varphi &=\text{ div} (m(\varphi) \nabla \mu), \ \text{ in } \  \Omega \times (0,T),\\
	\mu &= a \varphi - \J\ast \varphi + \F'(\varphi),\\
	\rho \u_t - 2 \text{div } ( \nu(\varphi) \mathrm{D}\u ) + (\u\cdot \nabla )\u + \nabla \uppi &= \mu \nabla \varphi + \mathbf{h}, \ \text{ in } \ \Omega \times (0,T),\\
	\text{div }\u&= 0, \ \text{ in } \  \Omega \times (0,T), \\
	\frac{\partial \mu}{\partial\mathbf{n}} &= 0 \ , \u=0 \ \text{ on } \ \partial \Omega \times [0,T],\\
	\u(0) &= \u_0, \  \ \varphi(0) = \varphi _0 \ \text{ in } \ \Omega, 
	\end{aligned}   
	\right.
	\end{equation}
	where $\Omega\subset \mathbb{R}^n,\ n=2,3$  and  $\u(x,t)$  and $\varphi(x,t)$ denote the average velocity of the fluid and  the relative concentration respectively. These equations are of the nonlocal type because of the presence of the term $\J$, which is the \emph{spatial-dependent internal kernel} and $\J \ast \varphi$ denotes the spatial convolution over $\Omega$. The mobility parameter is denoted by
	$m$, $\mu$ is the \emph{chemical potential}, $\uppi$ is the \emph{pressure}, $a$ is defined by $a(x) := \int _\Omega \J(x-y) \d y$, $\F$ is the configuration potential, which accounts for the presence of two phases, $\nu$ is the \emph{kinematic viscosity} and $\mathbf{h}$ is the external forcing term acting in the mixture.   The strain tensor $\mathrm{D}\u$ is the symmetric part of the gradient of the flow velocity vector, i.e., $\mathrm{D}\u$ is $\frac{1}{2}\left(\nabla\u+(\nabla\u)^{\top}\right)$. The chemical potential $\mu$ is the first variation of the free energy functional:
	\begin{align*}
	\mathcal{E}(\varphi) := \frac{1}{4}\int_\Omega \int_\Omega \mathrm{J}(x-y)(\varphi(x)-\varphi(y))^2\d x  \d y + \int_\Omega \mathrm{F}(\varphi(x))\d x.
	\end{align*} 
	Various simplified models of this system are studied by several mathematicians and physicists.   The local version of the system (see \cite{MR1700669, MR2580516})  is obtained by replacing  $\mu $ equation by 
	$ \mu =- \Delta \varphi + \F' (\varphi) $, which is the first variation of the free energy functional
	$$\mathcal{E}(\varphi) := \int_{\Omega} \left( \frac{1}{2} | \nabla \varphi(x)|^2 + \F(\varphi (x))\right)\, \d x.$$
	Another simplification appeared in the literature is to assume the constant mobility parameter and/or constant viscosity. The solvability, uniqueness and regularity of the system (\ref{1.1}) and of other simplified Cahn-Hilliard-Navier-Stokes models is well studied in the literature though most of the works are recent ones. Typically two types of potentials  are considered in the literature, regular potential as well as  singular potential. In general, singular potentials are difficult to handle and in such cases, $\F(\cdot)$ is usually approximated by polynomials in order to make the mathematical analysis easier (see \cite{MR2347608}). 
	
	The nonlocal Cahn-Hilliard-Navier-Stokes system with regular potential has been analysed by M. Grasseli et al. in \cite{MR2834896, MR3518604, MR3090070}, etc. Taking advantage of  the results for the regular potential, they have also studied in \cite{MR3019479}, the existence of weak solution of the system with singular potential. Furthermore, they proved the existence of the global attractor in 2D and trajectory attractor in 3D. Strong solutions for the nonlocal Cahn-Hilliard-Navier-Stokes system was discussed in \cite{MR3903266}. Uniqueness results for the same were established in \cite{MR3518604}. In \cite{MR3688414}, authors considered the nonlocal Cahn-Hilliard equation with singular potential and constant mobility and studied well-posedness and regularity results. Moreover, they established the strict separation property in dimension 2. Regularity results in case of degenerate mobility were studied in \cite{frigeri10regularity}. The local Cahn-Hilliard-Navier-Stokes system with singular free energies has been studied in \cite{MR2563636, MR1700669}. Further, along the application side, the optimal control of nonlocal Cahn-Hilliard-Navier-Stokes equations and robust control for local Cahn-Hilliard-Navier-Stokes equations have been addressed in the works \cite{MR3456388,MR4104524,MR3565933,MR3436705, MR4108622,MR4131779}, etc.

	Solvability results for the stationary nonlocal Cahn-Hilliard equations with singular potential were discussed in \cite{MR2108884}, whereas authors in \cite{MR2347608} proved the convergence to the equilibrium solution of Cahn-Hilliard system with logarithmic free energy. The existence of the equilibrium solution for the steady state Navier-Stokes equation is well known in the literature  and can be found in the book \cite{MR1318914}.   In \cite{MR3524178}, the authors discussed the existence of a weak solution to the stationary local Cahn-Hilliard-Navier-Stokes equations. The author in \cite{MR3555135} studied a coupled Cahn-Hilliard-Navier-Stokes model with delays in a two-dimensional bounded domains and discussed the asymptotic behaviour of the weak solutions and the stability of the stationary solutions.  In this work, our main aim is to study  the well-posedness of nonlocal steady  state system corresponding to the model described in  \eqref{1.1} in dimensions $2$ and $3$ and to examine the stability properties of this solution in dimension $2$ (for convex singular potentials).    
	
	Throughout this paper, we consider $\F$ to be a singular potential.  A typical example is the logarithmic potential:
	\begin{align}\label{2}
	\F (\varphi) = \frac{\theta}{2} ((1+ \varphi) \ln (1+\varphi) + (1-\varphi) \ln (1-\varphi)) - \frac{\theta_c}{2} \varphi^2 ,\quad \varphi \in (-1,1),
	\end{align} 
	where $\theta,\theta_c>0$. The logarithmic terms are related to the entropy of the system and note that $\F$ is convex if and only if $\theta\geq \theta_c$. If $\theta\geq \theta_c$ the mixed phase is stable and if $0<\theta<\theta_c$, the mixed phase is unstable and phase separation occurs. 
	To the best of our  knowledge, the solvability results for  stationary nonlocal  Cahn-Hilliard-Navier-Stokes equations is not available in the literature. In the current work, using techniques similar to the one developed in  \cite{MR3524178},  we resolve this issue.
	We prove the existence of a weak solution to the stationary nonlocal Cahn-Hilliard-Navier-Stokes system in dimensions 2 and 3 with variable mobility and viscosity. Further, we answer the questions of uniqueness and regularity of the solution for the equations with constant viscosity and  mobility parameters. In dimensions $2$ and $3$, we show that  the weak solution possesses higher regularity. The uniqueness of weak solutions is established  under certain conditions on the viscosity and mobility parameters. Lastly, for constant viscosity and mobility parameters and convex singular potential, we establish that the strong solution of  steady state equations in dimension $2$, stabilizes exponentially.
	
	The main difficulty to tackle while obtaining these results was to handle the nonlocal term. The nonlocal term in the equation needs careful estimation. Moreover, to the best of our knowledge, apart from existence and regularity results for steady state equation, which we have obtained here  is the first work to discuss exponential stability of nonlocal Cahn-Hilliard-Navier-Stokes model. These results can be  useful to study stabilisation properties of control problems associated with the system. 
	
	Rest of the paper is organised as follows: In the next section, we explain functional setting for the solvability of stationary nonlocal Cahn-Hilliard-Navier-Stokes equations \eqref{steadysys} (given below). We define the weak formulation of our system in section \ref{sec3}. The existence of a weak solution to the nonlocal Cahn-Hilliard-Navier-Stokes equations \eqref{steadysys} is proved using pseudo-monotonicity arguments and Browder's theorem in this section (see Theorem \ref{mainthm}).  In further studies, we assume the mobility parameter and viscosity to be constant. The section \ref{sec4} is devoted to study  the uniqueness of a weak solution for the system \eqref{steadysys} and some regularity results. We establish the uniqueness of weak solutions under certain assumptions on mobility parameter and viscosity (see Theorem \ref{unique-steady}). Further, we derive some regularity results for the solution. Finally, in section \ref{se4}, we establish that the stationary solution in two dimensions is exponentially stable (see Theorem \ref{thmexp}) under certain restrictions on mobility parameter and viscosity.
	\section{Stationary Nonlocal Cahn-Hilliard-Navier-Stokes System}\label{se3}\setcounter{equation}{0}
	In this section, we consider the stationary nonlocal Cahn-Hilliard-Navier-Stokes system in two and three dimensional bounded domains. Here, we consider the case of the coefficient of kinematic viscosity and mobility parameter depending on $\varphi$. Let us  consider the following steady state system associated with the equation \eqref{1.1}:
	\begin{equation}\label{steadysys}
	\left\{
	\begin{aligned} 
	\u\cdot \nabla \varphi &= \text{ div} (m(\varphi) \nabla \mu), \ \text{ in }\ \Omega,\\
	\mu &= a \varphi - \J\ast \varphi + \F'(\varphi), \ \text{ in }\ \Omega,\\
	- 2 \text{div } ( \nu(\varphi) \mathrm{D}\u )  + (\u\cdot \nabla )\u + \nabla \uppi &= \mu \nabla \varphi + \mathbf{h}, \  \text{ in }\ \Omega,\\
	\text{div }\u&= 0, \ \text{ in }\ \Omega, \\
	\frac{\partial \mu}{\partial\mathbf{n}} &= 0 \ , \u=\mathbf{0} \ \text{ on } \ \partial \Omega,\\
	\end{aligned} 
	\right.
	\end{equation}
	with average of $\varphi$ given by $$\frac{1}{|\Omega|}\int_\Omega \varphi (x) \d x = k \in (-1,1),$$ where $|\Omega|$ is the Lebesgue measure of $\Omega$.
	Our main aim in this work is to study the existence, uniqueness, regularity  and stability of the system \eqref{steadysys}. For solvability, we formulate the problem in an abstract setup and use the well known Browder's theorem to establish the existence of a weak solution to the system \eqref{steadysys}. We further study regularity, uniqueness and exponential stability of the system with constant viscosity and mobility parameters by establishing a-priori estimates and under certain conditions on viscosity and mobility. 
	
	\subsection{Functional setting} We first explain the  functional spaces needed to obtain our main results. Let us define
	\begin{align*}
	\G_{\mathrm{div}} &:= \Big\{ \u \in \mathrm{L}^2(\Omega;\R^n) : \text{div }\u=0,\  \u\cdot \mathbf{n} \big|_{\partial \Omega}=0   \Big\}, \\
	\V_{\mathrm{div}} &:= \Big\{\u \in \mathrm{H}^1_0(\Omega;\R^n): \text{div }\u=0\Big\},\\ \mathrm{H}&:=\mathrm{L}^2(\Omega;\R),\ \mathrm{V}:=\mathrm{H}^1(\Omega;\R),
	\end{align*}
	where $n=2,3$. Let us denote $\| \cdot \|$ and $(\cdot, \cdot)$, the norm and the scalar product, respectively, on both $\mathrm{H}$ and $\G_{\mathrm{div}}$. The duality between any Hilbert space $\X$ and its topological dual $\X'$ is denoted by $_{\X'}\langle \cdot,\cdot\rangle_{\X}$. We know that $\V_{\mathrm{div}}$ is endowed with the scalar product 
	$$(\u,\v)_{\V_{\mathrm{div}} }= (\nabla \u, \nabla \v)=2(\mathrm{D}\u,\mathrm{D}\v),\ \text{ for all }\ \u,\v\in\V_{\mathrm{div}}.$$ The norm on $\V_{\mathrm{div}}$ is given by $\|\u\|_{\V_{\mathrm{div}}}^2:=\int_{\Omega}|\nabla\u(x)|^2\d x=\|\nabla\u\|^2$. In the sequel, we use the notations $\mathbb{H}^m(\Omega):=\mathrm{H}^m(\Omega;\mathbb{R}^n)=\mathrm{W}^{m,2}(\Omega;\mathbb{R}^n)$ and  $\mathrm{H}^m(\Omega):=\mathrm{H}^m(\Omega;\mathbb{R})=\mathrm{W}^{m,2}(\Omega;\mathbb{R})$ for Sobolev spaces of order $m$.

	Let us also define  
	\begin{align*}
	\mathrm{L}^2_{(k)}(\Omega)& := \left\{ f \in \mathrm{L}^2(\Omega;\R) : \frac{1}{|\Omega |} \int_\Omega f(x)\d x = k \right\},\\
	\mathrm{H}^1_{(0)} (\Omega)&:= \mathrm{H}^1(\Omega;\R) \cap\mathrm{L}^2_{(0)}(\Omega)= \left\{ f \in \mathrm{H}^1(\Omega;\R) : \int_\Omega f(x)\d x = 0 \right\}\ ,\\ \mathrm{H}^{-1}_{(0)}(\Omega)&:=\mathrm{H}_{(0)}^1(\Omega)'.
	\end{align*} 
	Note that $\mathrm{L}^2_{(0)}(\Omega)$ is a Hilbert space equipped with the usual inner product in $\mathrm{L}^2(\Omega)$. 
	Since $\Omega$ is a bounded smooth domain and the average of $f$  is zero in $\mathrm{H}^1_{(0)} (\Omega)$, using the Poincar\'e-Wirtinger inequality (see Lemma \ref{poin} below), we have $\|f\|\leq C_{\Omega}\|\nabla f\|$, for all $f\in \mathrm{H}^1_{(0)} (\Omega)$. Using this fact, one can also show that $\mathrm{H}^1_{(0)} (\Omega)$ is a Hilbert space equipped with the inner product
	\begin{align*}
	(\varphi ,\psi)_{\mathrm{H}^1_{(0)}} = (\nabla \varphi, \nabla \psi), \ \text{ for all }\ \varphi, \psi \in \mathrm{H}^1_{(0)}(\Omega) .
	\end{align*} 
	We can prove the following dense and continuous embedding:
	\begin{align} \label{zeroembedding}
	\mathrm{H}^1_{(0)} (\Omega)\hookrightarrow \mathrm{L}^2_{(0)}(\Omega)\equiv  \mathrm{L}^2_{(0)}(\Omega)' \hookrightarrow \mathrm{H}^{-1}_{(0)}(\Omega).
	\end{align}
	Note that the embedding is compact (see for example, Theorem 1, Chapter 5, \cite{MR2597943}). The projection $\mathrm{P}_0 : \mathrm{L}^2(\Omega) \rightarrow \mathrm{L}^2_{(0)}(\Omega)$ onto $\mathrm{L}^2$-space with mean value zero is defined by 
	\begin{align} \label{P0}
	\mathrm{P}_0 f := f - \frac{1}{|\Omega |} \int_\Omega f(x)\d x
	\end{align}
	for every $f \in \mathrm{L}^2(\Omega)$.

	For every $f \in \mathrm{V}'$, we denote $\overline{f}$ the average of $f$ over $\Omega$, that is, $\overline{f} := |\Omega|^{-1} {_{\mathrm{V}'}}\langle f, 1 \rangle_{\mathrm{V}}$. 
	Let us also introduce the spaces (see \cite{MR3518604} for more details)
	\begin{align*}\mathrm{V}_0 &=\mathrm{H}_{(0)}^{1}(\Omega)= \{ v \in \mathrm{V} : \overline{v} = 0 \},\\
	\mathrm{V}_0' &=\mathrm{H}_{(0)}^{-1}(\Omega)= \{ f \in \mathrm{V}' : \overline{f} = 0 \},\end{align*}
	and the operator $\mathcal{A} : \mathrm{V} \rightarrow \mathrm{V}'$ is defined by
	\begin{align*}\,_{\mathrm{V}'}\langle \mathcal{A} u ,v \rangle_{\mathrm{V}} := \int_\Omega \nabla u(x) \cdot \nabla v(x) \d x, \  \text{for all } \ u,v \in \mathrm{V}.\end{align*}
	Clearly $\mathcal{A}$ is linear and it maps $\mathrm{V}$ into $\mathrm{V}_0'$ and its restriction $\mathcal{B}$ to $\mathrm{V}_0$ onto $\mathrm{V}_0'$ is an isomorphism.   
	We know that for every $f \in \mathrm{V}_0'$, $\mathcal{B}^{-1}f$ is the unique solution with zero mean value of the \emph{Neumann problem}:
	$$
	\left\{
	\begin{array}{ll}
	- \Delta u = f, \  \mbox{ in } \ \Omega, \\
	\frac{\partial u}{\partial\mathbf{n}} = 0, \ \mbox{ on } \  \partial \Omega.
	\end{array}
	\right.
	$$
	In addition, we have
	\begin{align} \!_{\mathrm{V}'}\langle \mathcal{A}u , \mathcal{B}^{-1}f \rangle_{\mathrm{V}} &= \!_{\mathrm{V}'}\langle f ,u \rangle_{\mathrm{V}}, \ \text{ for all } \ u\in \mathrm{V},  \ f \in \mathrm{V}_0' , \label{bes}\\
	\!_{\mathrm{V}'}\langle f , \mathcal{B}^{-1}g \rangle_{\mathrm{V}} &= \!_{\mathrm{V}'}\langle g ,\mathcal{B}^{-1}f \rangle_{\mathrm{V}} = \int_\Omega \nabla(\mathcal{B}^{-1}f)\cdot \nabla(\mathcal{B}^{-1}g)\d x, \ \text{for all } \ f,g \in \mathrm{V}_0'.\label{bes1}
	\end{align}
	Note that $\mathcal{B}$ can be also viewed as an unbounded linear operator on $\mathrm{H}$ with
	domain $\D(\mathcal{B}) = \left\{v \in \mathrm{H}^2(\Omega)\cap \mathrm{V}_0 : \frac{\partial v}{\partial\mathbf{n}}= 0\text{ on }\partial\Omega \right\}$. 
	
	Below, we give some facts about the elliptic regularity theory of Laplace operator $ B_N = -\Delta  + \mathrm{I} $ with Neumann boundary conditions. 
	
		\begin{lemma}[$ L^p $ regularity for Neumann Laplacian, \cite{SJL_1961-1962____A6_0}, Theorem 9.26, \cite{MR2759829}] \label{Lp_reg}
			Let us assume that $ u $ satisfies $ B_N u = f$ and $ \frac{\partial u}{\partial n} =0$ in weak sense. Then the following holds:
			\begin{itemize}
				\item [(i)] Let $ f \in (W^{1,p'}(\Omega))'$, with $ 1<p'<\infty $. Then $ u \in W^{1,p}(\Omega) $, where $ \frac{1}{p}+\frac{1}{p'}=1 $ and there exists a constant $ C>0 $ such that 
				\begin{align*}
				\|u\|_{W^{1,p}(\Omega)} \leq C \|f\|_{(W^{1,p'}(\Omega))'}.
				\end{align*}
				\item [(ii)] Let $f \in L^p(\Omega)$, with $1 < p < \infty$. Then, $u \in  W^{2,p} (\Omega), -\Delta u + u = f$ for a.e. $x \in \Omega, \frac{\partial u}{\partial n}= 0$ on $\partial \Omega$ in the sense of traces and there exists $C > 0$ such that
				\begin{align*}
				\|u\|_{W^{2,p}(\Omega)} \leq C \|f\|_{L^p(\Omega)}.
				\end{align*}
			\end{itemize}
		\end{lemma}

	\subsection{Linear and non-linear operators}
	Let us define the Stokes operator $\A : \D(\A)\cap  \G_{\mathrm{div}} \to \G_{\mathrm{div}}$. In the case of no slip boundary condition 
	$$\A=-\mathrm{P}\Delta,\ \D(\A)=\mathbb{H}^2(\Omega) \cap \V_{\mathrm{div}},$$ where $\mathrm{P} : \mathbb{L}^2(\Omega) \to \G_{\mathrm{div}}$ is the \emph{Helmholtz-Hodge orthogonal projection}. Note also that, we have
	$$\!_{\V_{\mathrm{div}}'}\langle\A\u, \v\rangle_{\V_{\mathrm{div}}} = (\u, \v)_{\V\text{div}} = (\nabla\u, \nabla\v),  \text{ for all } \ \u, \v \in \V_{\mathrm{div}}.$$
	It should also be noted that  $\A^{-1} : \G_{\mathrm{div}} \to \G_{\text{div }}$ is a self-adjoint compact operator on $\G_{\mathrm{div}}$ and by
	the classical \emph{spectral theorem}, there exists a sequence $\lambda_j$ with $0<\lambda_1\leq \lambda_2\leq \lambda_j\leq\cdots\to+\infty$
	and a family of $\mathbf{e}_j \in \D(\A),$ which is orthonormal in $\G_\text{div}$ and such that $\A\mathbf{e}_j =\lambda_j\mathbf{e}_j$. We know that $\u \in\G_{\mathrm{div}}$ can be expressed as $\u=\sum\limits_{j=1}^{\infty}\langle \u,\mathbf{e}_j\rangle \mathbf{e}_j,$ so that $\A\u=\sum\limits_{j=1}^{\infty}\lambda_j\langle \u,\mathbf{e}_j\rangle \mathbf{e}_j,$ for all $\u\in \D(\A)\subset \G_{\mathrm{div}}$. Thus, it is immediate that 
	\begin{align}
	\|\nabla\u\|^2=\langle \A\u,\u\rangle =\sum_{j=1}^{\infty}\lambda_j|\langle \u,\mathbf{e}_j\rangle|^2\geq \lambda_1\sum_{j=1}^{\infty}|\langle \u,\mathbf{e}_j\rangle|^2=\lambda_1\|\u\|^2.
	\end{align}

	For $\u,\v,\w \in \V_{\mathrm{div}}$, we define the trilinear operator $b(\cdot,\cdot,\cdot)$ as
	$$b(\u,\v,\w) = \int_\Omega (\u(x) \cdot \nabla)\v(x) \cdot \w(x)\d x=\sum_{i,j=1}^n\int_{\Omega}u_i(x)\frac{\partial v_j(x)}{\partial x_i}w_j(x)\d x,$$
	and the bilinear operator $\B$ from $\V_{\mathrm{div}} \times \V_{\mathrm{div}} $ into $\V_{\mathrm{div}}'$ is defined by,
	$$ \!_{\V_{\mathrm{div}}'}\langle \B(\u,\v),\w  \rangle_{\V_{\mathrm{div}}} := b(\u,\v,\w), \  \text{ for all } \ \u,\v,\w \in \V_\text{{div}}.$$
	An integration by parts yields, 
	\begin{equation}\label{2.7}
	\left\{
	\begin{aligned}
	b(\u,\v,\v) &= 0, \ \text{ for all } \ \u,\v \in\V_\text{{div}},\\
	b(\u,\v,\w) &=  -b(\u,\w,\v), \ \text{ for all } \ \u,\v,\w\in \V_\text{{div}}.
	\end{aligned}
	\right.\end{equation}
	For more details about the linear and non-linear operators, we refer the readers to \cite{MR0609732}. Let us now provide some important inequalities, which are used frequently in the paper. 
	\begin{lemma}[Gagliardo-Nirenberg inequality, Theorem 2.1, \cite{MR1230384}] \label{gn}
		Let $\Omega\subset\R^n$ and $\u\in\mathrm{W}^{1,p}_0(\Omega;\R^n)$, $p\geq 1$. Then for any fixed number $p,q\geq 1$, there exists a constant $C>0$ depending only on $n,p,q$ such that 
		\begin{align}\label{gn0}
		\|\u\|_{\mathbb{L}^r}\leq C\|\nabla\u\|_{\mathbb{L}^p}^{\theta}\|\u\|_{\mathbb{L}^q}^{1-\theta},\;\theta\in[0,1],
		\end{align}
		where the numbers $p, q, r, n$ and $\theta$ satisfy the relation
		$$\theta=\left(\frac{1}{q}-\frac{1}{r}\right)\left(\frac{1}{n}-\frac{1}{p}+\frac{1}{q}\right)^{-1}.$$
	\end{lemma}
	A particular case of Lemma \ref{gn} is the well known inequality due to Ladyzhenskaya (see Lemma 1 and 2, Chapter 1, \cite{MR0254401}), which is given below:
	\begin{lemma}[Ladyzhenskaya's inequality]\label{lady}
		For $\u\in\C_0^{\infty}(\Omega;\R^n), n = 2, 3$, there exists a constant $C$ such that
		\begin{align}\label{lady1}
		\|\u\|_{\mathbb{L}^4}\leq C^{1/4}\|\u\|^{1-\frac{n}{4}}\|\nabla\u\|^{\frac{n}{4}},\text{ for } n=2,3,
		\end{align}
		where $C=2,4,$ for $n=2,3$ respectively. 
	\end{lemma}
	Note that the above inequality is true even in unbounded domains. For $n=3$, $r=6$, $p=q=2$, from \eqref{gn0}, we find $\theta=1$ and 
	\begin{align*}
	\|\u\|_{\mathbb{L}^6}\leq C\|\nabla\u\|=C\|\u\|_{\V_{\mathrm{div}}}.
	\end{align*}
	For $n=2$,  the following estimate holds: 
	\begin{align*}
	|b(\u,\v,\w)| \leq \sqrt{2}\|\u\|^{1/2}\| \nabla \u\|^{1/2}\|\v\|^{1/2}\| \nabla \v\|^{1/2}\| \nabla \w\|,
	\end{align*}
	for every $\u,\v,\w \in \V_{\mathrm{div}}$. Hence, for all $\u\in\V_{\mathrm{div}},$ we have
	\begin{align}
	\label{be}
	\|\B(\u,\u)\|_{\V_{\mathrm{div}}'}\leq \sqrt{2}\|\u\|\|\nabla\u\|\leq \sqrt{\frac{2}{\lambda_1}}\|\u\|_{\V_{\mathrm{div}}}^2 ,
	\end{align}
	by using the Poincar\'e inequality. Similarly, for $n=3$, we have 
	\begin{align*}
	|b(\u,\v,\w)| \leq 2\|\u\|^{1/4}\| \nabla \u\|^{3/4}\|\v\|^{1/4}\| \nabla \v\|^{3/4}\| \nabla \w\|,
	\end{align*}
	for every $\u,\v,\w \in \V_{\mathrm{div}}$. Hence, for all $\u\in\V_{\mathrm{div}},$ using the Poincar\'e inequality, we have
	\begin{align}
	\|\B(\u,\u)\|_{\V_{\mathrm{div}}'}\leq 2\|\u\|^{1/2}\|\nabla\u\|^{3/2}\leq \frac{2}{\lambda_1^{1/4}}\|\u\|_{\V_{\mathrm{div}}}^2.
	\end{align}

	We also need the following general version of the Gagliardo-Nirenberg interpolation inequality to prove the regularity results. For functions $\u: \Omega\to\R^n$ defined on a bounded Lipschitz domain $\Omega\subset\R^n$, the Gagliardo-Nirenberg interpolation inequality is given by: 
	\begin{lemma}[Gagliardo-Nirenberg interpolation inequality, Theorem on Page125 , \cite{MR109940}]\label{GNI} Let $\Omega\subset\R^n$, $\u\in\mathrm{W}^{m,p}(\Omega;\R^n), p\geq 1$ and fix $1 \leq p,q \leq \infty$ and a natural number $m$. Suppose also that a real number $\theta$ and a natural number $j$ are such that
		\begin{align}
		\label{theta}
		\theta=\left(\frac{j}{n}+\frac{1}{q}-\frac{1}{r}\right)\left(\frac{m}{n}-\frac{1}{p}+\frac{1}{q}\right)^{-1}\end{align}
		and
		$\frac{j}{m} \leq \theta \leq 1.$ Then for any $\u\in\mathrm{W}^{m,p}(\Omega;\R^n),$ we have 
		\begin{align}\label{gn1}
		\|\nabla^j\u\|_{\mathbb{L}^r}\leq C\left(\|\nabla^m\u\|_{\mathbb{L}^p}^{\theta}\|\u\|_{\mathbb{L}^q}^{1-\theta}+\|\u\|_{\mathbb{L}^s}\right),
		\end{align}
		where $s > 0$ is arbitrary and the constant $C$ depends upon the domain $\Omega,m,n$. 
	\end{lemma}
	If $1 < p < \infty$ and $m - j -\frac{n}{p}$ is a non-negative integer, then it is necessary to assume also that $\theta\neq 1$. Note that  for $\u\in\mathrm{W}_0^{1,p}(\Omega;\R^n)$, Lemma \ref{gn} is a special case of the above inequality, since for $j=0$, $m=1$ and $\frac{1}{s}=\frac{\theta}{p}+\frac{1-\theta}{q}$ in \eqref{gn1}, and application of the Poincar\'e inequality yields (\ref{gn0}). It should also be noted that \eqref{gn1} can also be written as 
	\begin{align}\label{gn2}
	\|\nabla^j\u\|_{\mathbb{L}^r}\leq C\|\u\|_{\mathbb{W}^{m,p}}^{\theta}\|\u\|_{\mathbb{L}^q}^{1-\theta},
	\end{align}
	By taking $j=1$, $r=4$, $n=m=p=q=s=2$ in (\ref{theta}), we get $\theta=\frac{3}{4},$ and from \eqref{gn1} we get 
	\begin{align}\label{gu}
	\|\nabla\u\|_{\mathbb{L}^4}\leq C\left(\|\Delta\u\|^{3/4}\|\u\|^{1/4}+\|\u\|\right).
	\end{align} 
	Also, taking $j=1$, $r=4$, $n=3$, $m=p=q=s=2$ in (\ref{theta}), we get $\theta=\frac{7}{8},$ and 
	\begin{align}\label{gua}
	\|\nabla\u\|_{\mathbb{L}^4}\leq C\left(\|\Delta\u\|^{7/8}\|\u\|^{1/8}+\|\u\|\right).
	\end{align} 
	
	\begin{lemma}
		[Poincar\'e-Wirtinger inequality, Corollary 12.28, \cite{MR3726909}]\label{poin}
		Assume that $1 \leq p <  \infty$ and that $\Omega$ is a bounded connected open subset of the $n$-dimensional Euclidean space $\mathbb{R}^n$ whose boundary is of class $\mathrm{C}$. Then there exists a constant $C_{\Omega,p}>0$, such that for every function $\phi \in \mathrm{W}^{1,p}(\Omega)$,
		$$\|\phi-\overline{\phi}\|_{\mathrm{L}^{p}(\Omega )}\leq C_{\Omega,p}\|\nabla \phi\|_{\mathrm{L}^{p}(\Omega )},$$
		where
		$\overline{\phi}={\frac {1}{|\Omega |}}\int _{\Omega }\phi(y)\,\mathrm {d} y$
		is the average value of $\phi$ over $\Omega$.
	\end{lemma}
	
	\subsection{Basic assumptions}
	Let us now make the following assumptions on $\mathrm{J}$ and $\mathrm{F}$ in order to establish the solvability results of the system \eqref{steadysys}. We suppose that the potential $\F$ can be written in the following form $$\F= \F_1 + \F_2$$ where $\F_1 \in \C^{(2+2q)}(-1,1)$ with $q \in \mathbb{N}$ fixed, and $\F_2 \in \C^2(-1,1)$.
	
	Now, we list the assumptions on $\nu, \J, \F_1, \F_2 $ and mobility $m$ (cf. \cite{MR3019479}).
	\begin{itemize}
		\item[(A1)] $ \J \in \W^{1,1}(\mathbb{R}^2;\R), \  \J(x)= \J(-x) \; \text {and} \ a(x) = \int_\Omega \J(x-y)\d y \geq 0,$ a.e., in $\Omega$. 
		\item[(A2)] The function $\nu$ is locally Lipschitz on $\mathbb{R}$ and there exists $\nu_1, \nu_2 >0$ such that 
		$$ \nu_1 \leq \nu(s) \leq \nu_2, \ \text{ for all }\  s \in \mathbb{R}.$$
		\item[(A3)] There exist $C_1>0$ and $\epsilon_0 >0$ such that 
		$$\F_1^{(2+2q)} (s) \geq C_1, \ \text{ for all } \ s \in (-1,1+\epsilon_0] \cup [1-\epsilon_0,1).$$
		
		\item[(A4)] There exists $\epsilon_0 >0$ such that, for each $k=0,1,...,2+2q$ and each $j=0,1,...,q$, 
		$$ \F_1^{(k)}(s) >0, \  \text{ for all }\ s \in [1-\epsilon_0,1)$$ 
		$$ \F_1^{(2j+2)}(s) \geq 0, \ \F_1^{(2j+1)}(s) \leq 0, \ \text{ for all }\ s \in (-1,1+\epsilon_0].$$
		\item[(A5)] There exists $\epsilon_0 >0$ such that $\F_1^{(2+2q)}$ is non-decreasing in $[1-\epsilon_0,1)$ and non-increasing in $(-1,-1+\epsilon_0]$.
		\item[(A6)] There exists $\alpha , \beta \in \mathbb{R}$ with $\alpha + \beta > - \min\limits_{[-1,1]} \F''_2$ such that 
		$$ \F_1''(s) \geq \alpha, \  \text{ for all }\ s \in (-1,1), \quad a(x) \geq \beta, \ \text{ a.e. }\ x \in \Omega.$$ 
		\item[(A7)] $\lim\limits_{s \rightarrow -1} \F_1'(s) = -\infty$ and $\lim\limits_{s \rightarrow 1} \mathrm{F}_1'(s) = \infty$.
		\item[(A8)] The function $m$ is locally Lipschitz function on $\mathbb{R}$ and there exists $m_1 , m_2 >0$ such that
		\begin{align*}
		m_1 \leq m(s) \leq m_2 , \ \text{ for all }\ s \in \mathbb{R}.
		\end{align*}
	\end{itemize}

	\begin{remark} \cite{MR3000606}
		We can represent the potential $\mathrm{F}$ as a quadratic perturbation of a convex function. That is  
		\begin{align}\label{decomp of F}
		\mathrm{F}(s)= \mathrm{G} (s) - \frac{\kappa}{2} s^2,
		\end{align}
		where $\mathrm{G}$ is strictly convex and $\kappa>0$.
	\end{remark}
	We further assume that   
	\begin{itemize}
		\item[(A9)] there exists $C_0 > 0$ such that  $\F''(s) + a(x) \geq C_0$ , for all $s \in (-1,1)$ a.e. $x \in  \Omega$ and 
		$  \|\J \|_{\L^1} \leq {C_0}  + \kappa$.
	\end{itemize}
	
	\begin{remark}\label{remark J}
		Assumption $\J \in \W^{1,1}(\mathbb{R}^2;\R)$ can be weakened. Indeed, it can be replaced by $\J \in \W^{1,1}(\B_\delta;\R)$, where $\B_\delta := \{z \in \mathbb{R}^2 : |z| < \delta \}$ with $\delta := \emph{diam}(\Omega)=\sup\limits_{x,y\in \Omega}d(x,y)$, where $d(\cdot,\cdot)$ is the Euclidean metric on $\mathbb{R}^2$, or also by
		\begin{eqnarray} \label{Estimate J}
		\sup_{x\in \Omega} \int_\Omega \left( |\J(x-y)| + |\nabla \J(x-y)| \right) \d y < +\infty.
		\end{eqnarray}
	\end{remark}
	Note that \eqref{Estimate J} says that $\sup\limits_{x \in \Omega} \|\J\|_{\W^{1,1}} $ is finite and hence assumption (A1) is justified.
	
	\begin{remark}\label{remark F}
		Assumptions (A3)-(A6) are satisfied in the case of the physically relevant logarithmic double-well potential \eqref{2}, for any fixed positive integer $q$. In particular setting 
		\begin{align*}
		\F_1 (s)= \frac{\theta}{2}((1+s)\ln(1+s)+(1-s)\ln(1-s)), \qquad \F_2(s) =- \frac{\theta_cs^2}{2}
		\end{align*}
		then Assumption (A6) is satisfied iff $\beta > \theta_c - \theta.$   
	\end{remark}
		
	\section{Existence of Weak Solution}\label{sec3}\setcounter{equation}{0}
	In this section, we establish the existence of a weak solution to the system \eqref{steadysys} using pseudo-monotonicity arguments and  Browder's theorem. Let us first give the definition of \emph{weak solution} of the system \eqref{steadysys}.
	\begin{definition} \label{weakdef}
		Let $\mathbf{h} \in \V'_{{\mathrm{div}}}$ and $k\in (-1,1)$ be fixed.	A triple $(\mathbf{u},\mu , \varphi) \in \V_{{\mathrm{div}}} \times \mathrm{V} \times (\mathrm{V} \cap \L^2_{(k)}(\Omega) ) $ is called a \emph{weak solution} of the problem \eqref{steadysys} if
		\begin{align} 
		\label{weakphi}    \int_\Omega (\mathbf{u} \cdot \nabla \varphi)  \psi \, \d x =& -\int_\Omega m(\varphi) \nabla \mu \cdot \nabla \psi \, \d x, \\
		\label{weakmu}  \int_\Omega \mu \psi \, \d x=& \int_\Omega (a \varphi -\J * \varphi) \psi \, \d x +\int_\Omega \F'(\varphi) \psi \, \d x, \\
		\label{weakform nse}
		\int_\Omega (\mathbf{u} \cdot \nabla) \mathbf{u} \cdot \mathbf{v} \, \d x+ \int_\Omega 2\nu (\varphi)  \mathrm{D}\mathbf{u} \cdot  \mathrm{D}\mathbf{v}\, \d x =& \int_\Omega\mu \nabla \varphi \cdot \mathbf{v} \, \d x + \langle \mathbf{h} , \mathbf{v} \rangle,
		\end{align}
		for every $\psi \in \mathrm{V}$ and $\v \in \V_{{\mathrm{div}}}$.
	\end{definition}
	Our aim is to establish the existence of a weak solution of the system \eqref{steadysys} in the sense of Definition \ref{weakdef}. But, when working with the above definition, some difficulties arise in the analysis of our problem. The most important one is that $\mathrm{L}^2_{(k)}(\Omega)$ is not a vector space for $k\neq 0$. But, we can assume that $k=0$ with out loss of generality. Otherwise replace $\varphi$ by $\widetilde{\varphi}:= \varphi - k$ and $\F$ by $\F_k$ with $\F_k(x) := \F(x+k)$ for all $x \in \mathbb{R}$. Thus, in order to establish a weak solution of the system \eqref{steadysys}, we first reformulate the problem \eqref{weakphi}-\eqref{weakform nse}. We prove the existence of a solution to the reformulated problem \eqref{reformphi}-\eqref{nsreform} (see below) instead of \eqref{weakphi}-\eqref{weakform nse}.  We establish the equivalence of these two problems in Lemma \ref{reformtoweak}.
	
	We reduce $\mu$  to $\mu_0$, which has mean value $0$,  and it would help in proving  coercivity of an operator in the later part of this section. Let us fix $ \mu_0 =\mu - \frac{1}{|\Omega |} \int_\Omega \F'(\varphi)\d x$.  Then the reformulated problem of \eqref{weakphi}-\eqref{weakform nse} is given by 
	\begin{align} 
	\label{reformphi} \int_\Omega (\mathbf{u} \cdot \nabla \varphi) \psi \,\d x =& -\int_\Omega m(\varphi) \nabla \mu_0 \cdot \nabla \psi\, \d x, \\ 
	\label{mu-reform} \int_\Omega \mu_0 \psi \,\d x=& \int_\Omega (a \varphi -\J * \varphi) \psi \,\d x +\int_\Omega \mathrm{P}_0 (\F'(\varphi)) \psi \,\d x, \\
	\label{nsreform}
	\int_\Omega (\mathbf{u} \cdot \nabla)\mathbf{u} \cdot\mathbf{v} \,\d x+ \int_\Omega 2\nu (\varphi) \mathrm{D}\mathbf{u} \cdot \mathrm{D}\mathbf{v}\, \d x =& \int_\Omega\mu_0 \nabla \varphi \cdot \mathbf{v} \,\d x + \langle \mathbf{h} , \mathbf{v} \rangle,
	\end{align}
	where $\mathbf{u} \in \V_{\mathrm{div}} , \mu_0 \in \mathrm{V}_0, \varphi \in \mathrm{V}_0$.  Now we show that proving the existence of a solution to the equations \eqref{reformphi}-\eqref{nsreform} would also give a solution to \eqref{weakphi}-\eqref{weakform nse}.
	\begin{lemma}\label{reformtoweak}
		Let $(\mathbf{u},\mu_0,\varphi) \in\V_{{\mathrm{div}}}\times \mathrm{V}_0 \times \mathrm{V}_0$ be a solution to the system \eqref{reformphi}-\eqref{nsreform}. Then $(\mathbf{u},\mu,\varphi)$ is a solution to the weak formulation \eqref{weakphi}-\eqref{weakform nse}, where $ \mu_0 =\mu - \frac{1}{|\Omega |} \int_\Omega \F'(\varphi)\d x$. 
	\end{lemma}
	\begin{proof}
		Let $(\mathbf{u},\mu_0, \varphi) \in \V_{\mathrm{div}}\times \mathrm{V}_0 \times \mathrm{V}_0$ be a solution of the system \eqref{reformphi}-\eqref{nsreform}. Let $\overline{\mu}= \frac{1}{|\Omega|} \int_\Omega \F'(\varphi)\d x$. Since $\overline{\mu}$ is a scalar, from \eqref{nsreform},  using integration by parts, one  can easily deduce that  $ \int_\Omega \overline{\mu} \nabla \varphi \cdot \mathbf{v} \d x=0 $.  Then, we have 
		\begin{align*}
		\int_\Omega (\mathbf{u} \cdot \nabla) \mathbf{u} \cdot \mathbf{v} \d x+ \int_\Omega 2\nu (\varphi) \mathrm{D}\mathbf{u} \cdot \mathrm{D}\mathbf{v} \d x =& \int_\Omega\mu_0 \nabla \varphi \cdot \mathbf{v} \d x + \langle \mathbf{h} , \mathbf{v} \rangle+ \int_\Omega \overline{\mu} \nabla \varphi \cdot \mathbf{v} \d x \\
		=& \int_\Omega\mu \nabla \varphi \cdot \mathbf{v} \d x + \langle \mathbf{h} , \mathbf{v} \rangle , 
		\end{align*}
		which gives the equation \eqref{weakform nse}. Once again, since  $\overline{\mu}$ is a scalar quantity, we can clearly see \eqref{weakphi} follows from \eqref{reformphi}. Now it is left to prove \eqref{weakmu}. From \eqref{mu-reform}, we have 
		\begin{align} \label{3.3}
		\int_\Omega \mu_0 \psi \d x=& \int_\Omega (a \varphi -\J * \varphi) \psi \d x +\int_\Omega \mathrm{P}_0 (\F'(\varphi)) \psi \d x.
		\end{align}
		Using \eqref{P0} and substituting value of $\mu_0$ in \eqref{3.3} we get, 
		\begin{align*}
		\int_\Omega \mu \psi \d x=& \int_\Omega (a \varphi -\J * \varphi) \psi \d x +\int_\Omega (\mathrm{P}_0 (\F'(\varphi))+ \overline{\mu} )\psi \d x \\
		=& \int_\Omega (a \varphi -\J * \varphi) \psi \d x +\int_\Omega \F'(\varphi)\psi \d x,
		\end{align*}
		which completes the proof. 
	\end{proof}

	\subsection{Preliminaries}
	In order to formulate the problem \eqref{reformphi}-\eqref{nsreform}  in the framework of Browder's Theorem (see Theorem \ref{Browder} below), we need some preliminaries, which we state below.
	
	Let $\X$ be a Banach space and $\X'$ be its topological dual. Let $\T$ be a function from $\X$ to $\X'$ with domain $\D=\mathrm{D}(\T) \subseteq \X.$
	\begin{definition}[Definition 2.3, \cite{MR3014456}] The function $\T$ is said to be
		\begin{itemize} 
			\item[(i)]  \emph{demicontinuous}  if for every sequence $u_k \in \D, u \in \D$ and $u_k \rightarrow u$ in $\X$ implies that $\T(u_k) \rightharpoonup \T(u)$ in $\X'$,
			\item[(ii)] \emph{hemicontinuous}  if for every $u \in \D, v \in \X$ and for every sequence of positive real numbers $t_k$ such that $u+t_k v \in \D,$ it holds that $t_k \rightarrow 0$ implies $\T(u+t_k v) \rightharpoonup \T(u)$ in $\X'$,
			\item[(iii)]  \emph{locally bounded} if for every sequence $u_k \in \D$, $u \in \D$ and $u_k \rightarrow u $ in $\X$ imply that $\T(u_k)$ is bounded in $\X'$.
		\end{itemize}
		
	\end{definition}
	From the above definition, it is clear that a demicontinuous function is hemicontinuous and locally bounded. 
	\begin{definition}[Definition 2.1 (iv), \cite{MR3014456}]
		We say that $\T$ is \emph{pseudo-monotone} if, for every sequence $u_k$  in $\X$  such that $u_k \rightharpoonup u$ in $\X$ and 
		\begin{equation*}
		\limsup_{k \rightarrow \infty}  \,_{\X'}\langle \T(u_k), u_k - u\rangle_{\X} \leq 0
		\end{equation*} 
		implies 
		\begin{equation*}
		\liminf_{k \rightarrow \infty} \,_{\X'}\langle \T(u_k), u_k-v \rangle_{\X} \geq \,_{\X'}\langle \T(u),u-v \rangle_{\X}
		\end{equation*}
		for every $v \in \X$. Moreover $\T$ is said to be \emph{monotone} if  $$ \,_{\X'}\langle \T(u)-\T(v),u-v\rangle_{\X} \geq 0, \ \text{ for every }\ u ,v\in \D.$$
	\end{definition}
	\begin{definition}
		A mapping $\T: \X \rightarrow \X'$ is said to be maximal monotone if it is monotone and its graph
		\begin{align*}
		\mathrm{G}(\T)= \left\{ (u,w) : w \in \T(u) \right\} \subset \X \times \X'
		\end{align*}
		is not properly contained in graph of any other monotone operator. In other words,  for $u\in\mathbb{X}$ and $w\in\mathbb{X}'$, the inequality $\,_{\X'}\langle w-\mathrm{T}(v),u-v\rangle_{\X}\geq 0$, for all $v\in\mathbb{X}$ implies $w=\mathrm{T}(u)$.
	\end{definition}
	
	\begin{definition}[Definition 2.3, \cite{MR3014456}]
		Let $\X$ and $\Y$ be Banach spaces. A bounded linear operator $\T:\X \rightarrow \Y$ is said to be a \emph{completely continuous operator} if for every sequence $u_k \rightharpoonup u$ in $\X$ implies $\T(u_k) \rightarrow \T(u)$ in $\Y$. 
	\end{definition}
	We can see that complete continuity implies pseudo-monotonicity (Corollary 2.12, \cite{MR3014456}).
	
	\begin{lemma}[Lemma 5.1, \cite{MR3524178} ] \label{pseudo-monotone}
		Let $\X$ be a real, reflexive Banach space and $\widetilde{\T} : \X \times \X \rightarrow \X'$ be such that for all $u \in \X$:
		\begin{enumerate}
			\item $\widetilde{\T} (u, \cdot): \X \rightarrow \X'$ is monotone and hemicontinuous.
			\item $\widetilde{\T} ( \cdot,u): \X \rightarrow \X'$ is completely continuous.
		\end{enumerate}
		Then the operator $\T:\X \rightarrow \X'$ defined by $\T(u) := \widetilde{\T}(u,u)$ is pseudo-monotone.
	\end{lemma}
	
	\begin{definition}
		Let $\X$ be a real Banach space and $f:\X \rightarrow (-\infty, \infty]$ be a functional on $\X$. A linear functional $g \in \X'$ is called \emph{subgradient} of $f$ at $u$ if $f(u) \not\equiv \infty$ and 
		\begin{align*}
		f(v) \geq f(u) + \!_{\X'}\langle g, v-u \rangle_{\X},
		\end{align*}
		holds for all $v \in \X$.
	\end{definition}
	We know that subgradient of a functional need not be unique. The set of all subgradients of $f$ at $u$ is called \emph{subdifferential} of $f$ at $u$ and is denoted by  $\partial f (u)$.We say that $f$ is \emph{Gateaux differentiable} at $u$ in $\X$ if $\partial f(u)$ consists of exactly one element (see \cite{MR1195128}).
	
	\begin{lemma}[Theorem A, \cite{MR262827}] \label{rockfellar}
		If $f$ is a lower semicontinuous, proper convex function on $\X$ (that is, $f$ is a convex function  and $f$ takes values in the extended real number line such that $f(u)<+\infty $
		for at least one $u\in\X$ and $f(u)>-\infty $
		for every $u\in\X$), then $\partial f$ is a maximal monotone operator from $\X$ to $\X'$.
	\end{lemma}

	Now we state Browder's theorem, which we use to prove the existence of solution to the problem \eqref{reformphi}-\eqref{nsreform}.
	\begin{theorem}[Theorem 32.A. in \cite{MR1033498}, Browder] \label{Browder} 
		\hfill
		\begin{enumerate}
			\item Let $\Y$ be a non-empty, closed and convex subset of a real and reflexive Banach space $\X$.
			\item Let $\T: \Y \rightarrow \mathcal{P}(\X')$ be a maximal monotone operator, where  $\mathcal{P}(\X')$ denotes the power set of $\X'$. 
			\item Let $\S: \Y \rightarrow \X'$ be a pseudo-monotone, bounded and demicontinuous mapping.
			\item If the set $\Y$ is unbounded, then the operator $\S$ is $\T$-coercive with respect to the fixed element $b \in \X'$, that is, there exists an element $u_0 \in \Y \cap \mathrm{D} (\T)$ and $R>0$ such that 
			\begin{align*}
			\!_{\X'}\langle \S(u), u-u_0 \rangle_{\X} > \,_{\X'}\langle b, u-u_0 \rangle_{\X},
			\end{align*}
			for all $u \in \Y$ with $\| u\|_{\X}>R$.
		\end{enumerate}Then the problem 
		\begin{align}
		b\in \T (u) + \S(u)
		\end{align}
		has a solution $u \in \Y \cap \mathrm{D}(\T)$.
	\end{theorem}

	\subsection{The functional $f$} We mainly follow the work of \cite{MR3524178} (local Cahn-Hilliard-Navier-Stokes equations) to establish the solvability results of the system \ref{steadysys}. Before we proceed to prove our main result, we first consider the following functional and study its properties. Let us define 
	\begin{align}
	\label{functional} f(\varphi):= \frac{1}{4}\int_\Omega \int_\Omega \J(x-y)(\varphi(x)-\varphi(y))^2\d x \d y + \int_\Omega \mathrm{G}(\varphi(x))\d x,
	\end{align}
	where $\varphi \in \L^2_{(0)}(\Omega)$ and $\mathrm{G}$ is given in \eqref{decomp of F}. Using assumption \ref{decomp of F}, we define $\mathrm{G}(x)= \infty$ for $x \notin  [-1, 1]$. The domain of $f$ is given by
	\begin{align} \label{domain0f_f}
	\mathrm{D}(f) = \left\{ \varphi \in  \L^2_{(0)}(\Omega) :  \mathrm{G}(\varphi) \in \L^1(\Omega) \right\}.
	\end{align}
	For $\varphi \notin \mathrm{D}(f)$, we set $f(\varphi)= +\infty$. 
	Note that $ \mathrm{D}(f) \neq \varnothing$.
	
	Given a functional $f:\L^2_{(0)}(\Omega) \rightarrow (-\infty,\infty]$,   its subgradient maps from $\L^2_{(0)}(\Omega)$ to $ \mathcal{P}{(\L^2_{(0)}(\Omega)')}$. We write $\partial_{\L^2_{(0)}}f$ as the subgradient  of the functional $f:\L^2_{(0)}(\Omega) \rightarrow (-\infty,\infty]$. Since $\mathrm{V}_0 \hookrightarrow\mathrm{L}^2_{(0)}(\Omega)$, we can also consider $f$ as a functional from $\mathrm{V}_0$
	to $(-\infty,\infty]$, and hence we have to distinguish between its different subgradients. If we consider $f:\mathrm{V}_0 \rightarrow (-\infty,\infty]$, then the subgradient of $f$ is denoted by $\partial_{\mathrm{V}_0} f$.	
	
	\begin{proposition} \label{gatuexderivative}
		Let $f$ be defined as in \eqref{functional}. Then $f$ is Gateaux differentiable on $\L^2_{(0)} (\Omega)$ and we have 
		\begin{align}\label{sub}
		\partial_{\L^2_{(0)}} f(\varphi) = a \varphi -\J*\varphi + \mathrm{P}_0 \mathrm{G}'(\varphi).
		\end{align}
		and 
		\begin{align*}
		\mathrm{D}(\partial_{\L^2_{(0)}}f) = \left\{ \varphi \in \L^2_{(0)}(\Omega) : \mathrm{G}'(\varphi) \in \mathrm{H} \right\}.
		\end{align*}
		Furthermore, it holds that 
		\begin{align}\label{3.10}
		\|\partial_{\L^2_{(0)}}f(\varphi)\|\leq \left(\|a\|_{\mathrm{L}^{\infty}}+\|\J\|_{\mathrm{L}^1}\right)\|\varphi\|+\|\mathrm{G}'(\varphi)\|\leq 2a^*\|\varphi\|+\|\mathrm{G}'(\varphi)\|. 
		\end{align}
	\end{proposition}
	\begin{proof}
		Let $h \in \L^2_{(0)}(\Omega)$. Then, we have
		\begin{align*}
			&	\frac{f(\varphi + \epsilon h)-f(\varphi)}{\epsilon} \nonumber\\&= \frac{1}{\epsilon} \left[\frac{1}{4}\int_\Omega \int_\Omega \J(x-y)((\varphi+ \epsilon h)(x)-(\varphi+\epsilon h)(y))^2\d x \d y + \int_\Omega \mathrm{G}((\varphi+\epsilon h)(x))\d x  \right. \\
			& 	\hspace{2cm}\left. -\left(\frac{1}{4}\int_\Omega \int_\Omega \J(x-y)((\varphi)(x)-\varphi(y))^2\d x \d y + \int_\Omega \mathrm{G}(\varphi(x))\d x \right) \right] \\
			&= \frac{1}{4}\int_\Omega \int_\Omega \J(x-y)(2(\varphi(x)-\varphi(y))(h(x)-h(y)) + \epsilon (h(x)-h(y))^2) \\
			& \hspace{1cm}+ \frac{1}{\epsilon} \int_\Omega (\mathrm{G}((\varphi+\varepsilon h)(x)) - \mathrm{G}(\varphi(x)) )\d x .
			\end{align*}	
		Hence 
		\begin{align}
		\lim_{\epsilon \rightarrow 0 }\frac{f(\varphi + \epsilon h)-f(\varphi)}{\epsilon} = (a \varphi- \J*\varphi + \mathrm{G}'(\varphi), h).
		\end{align} 
	
		Since $\mathrm{P}_0$ is the orthogonal projection onto $\L^2_{(0)}(\Omega)$, we get \eqref{sub}. Note that \eqref{3.10} is an immediate consequence of \eqref{sub}. 
	\end{proof}

	\begin{lemma} \label{lowersemi}
		The functional $f: \L^2_{(0)}(\Omega) \rightarrow (-\infty,\infty]$ defined by \eqref{functional} is proper convex and lower semicontinuous with $\mathrm{D}(f) \neq \varnothing$.
	\end{lemma}
	\begin{proof}
		\textbf{Claim (1).} \emph{$f$ is proper convex with $\mathrm{D}(f) \neq \varnothing$:} 	In order to prove convexity, we use the fact that $f$ is convex if and only if $$(\partial_{\L^2_{(0)}} f(\varphi)-\partial_{\L^2_{(0)}} f(\psi) , \varphi - \psi) \geq  0,$$ for all $\varphi,\psi\in\L^2_{(0)}(\Omega)$.  Observe that for $0<\theta<1$, using Taylor's series expansion and the definition of $\mathrm{P}_0$, we find that
		\begin{align*}
		&\hspace*{-0.3 cm}(a \varphi - \J*\varphi+ \mathrm{P}_0 \mathrm{G}'(\varphi) -(a \psi - \J*\psi+  \mathrm{P}_0 \mathrm{G}'(\psi) ), \varphi-\psi) \qquad \qquad \\
		&=(a \varphi - \J*\varphi+ \mathrm{F}'(\varphi) +\kappa \varphi-(a \psi - \J*\psi+  \mathrm{F}'(\psi)+\kappa\psi ), \varphi-\psi) \qquad \qquad \\
		&= (a(\varphi-\psi)-\J*(\varphi-\psi)+ \F'(\varphi)- \F'(\psi)+ \kappa(\varphi- \psi) ,\varphi-\psi) \qquad \ \\
		&= (a(\varphi-\psi)-\J*(\varphi-\psi)+ \F''(\varphi + \theta \psi)(\varphi- \psi)+ \kappa(\varphi- \psi) ,\varphi-\psi)  \\
		&\geq C_0 \| \varphi- \psi\|^2 - \|\J\|_{\mathrm{L}^1} \| \varphi- \psi\|^2 + \kappa \| \varphi- \psi\|^2 \quad \hspace{3.65cm} \\
		&\geq (C_0 + \kappa - \|\J\|_{\mathrm{L}^1})\| \varphi- \psi\|^2 \hspace{7.78cm} \\
		&\geq 0, 
		\end{align*}
		using (A9). Hence $f$ is a convex functional  on $\L^2_{(0)}(\Omega) $. Since the domain of $f$ is 	$\mathrm{D}(f) = \left\{ \varphi \in  \L^2_{(0)}(\Omega) :  \mathrm{G}(\varphi) \in \L^1(\Omega) \right\} \neq \varnothing,$ it is immediate that $f$ is proper convex.
		
		\textbf{Claim (2).} \emph{$f$ is lower semicontinuous:}   Let $\varphi_k \in \L^2_{(0)}(\Omega)$ and $\varphi_k \rightarrow \varphi $ in $\L^2_{(0)}(\Omega)$ as $k \rightarrow \infty$. Our aim is to establish that $	f(\varphi)\leq \liminf\limits_{k \rightarrow \infty} f(\varphi_k).$  It is enough to consider the case $\liminf\limits_{k \rightarrow \infty} f(\varphi_k) < +\infty$. Thus, for this sequence, we can assume that $f(\varphi_k) \leq \M$, for some $\M$. This  implies that $\varphi_k \in \mathrm{D}(f)$, for all $k\in\mathbb{N}$. Since $\mathrm{G}:[-1,1] \rightarrow \mathbb{R}$ is a continuous function, by adding a suitable constant, we can assume without loss of generality that $\mathrm{G} \geq 0$ and we have
		\begin{align*}
		\mathrm{G}(\varphi) \leq  \liminf_{k \rightarrow \infty} \mathrm{G}(\varphi_k),
		\end{align*}
		which gives from Fatou's lemma that
		\begin{align} \label{lowerG}
		\int_\Omega \mathrm{G}(\varphi(x)) \d x \leq \liminf_{k \rightarrow \infty} \int_\Omega \mathrm{G}(\varphi_k (x)) \d x .
		\end{align}
		Now consider the functional $\I(\cdot)$ defined by
		\begin{align*}
		\I(\varphi) :=\int_\Omega \int_\Omega \J(x-y)(\varphi(x)-\varphi(y))^2\d x  \d y = (a \varphi, \varphi) -(\J*\varphi, \varphi).
		\end{align*}
		We show that the function $\I(\cdot)$ is continuous. We consider 
		\begin{align*}
		|\I(\varphi_k)-\I(\varphi)|=&|(a \varphi_k, \varphi_k)-(\J*\varphi_k, \varphi_k) - (a \varphi, \varphi) -(\J*\varphi, \varphi)| \\
		\leq & |(a(\varphi_k- \varphi), \varphi_k )|+|(a \varphi, \varphi_k-\varphi)| + |(\J*(\varphi_k-\varphi), \varphi_k)| + |(\J*\varphi,\varphi_k-\varphi)|\\
		\leq & \|a\|_{\L^\infty} \,(\|\varphi_k\| + \| \varphi\|) \|\varphi_k-\varphi\| + \|\J\|_{\L^1}(\|\varphi_k\| + \| \varphi\|) \|\varphi_k-\varphi\|,
		\end{align*}
		where we used Young's inequality, Young's inequality for convolution and H\"older's inequality. Then, we have 	$|\I(\varphi_k)-\I(\varphi)|\to 0 \ \ \text{as } k \rightarrow \infty  $,	since $\varphi_k \rightarrow \varphi $  in $\L^2_{(0)}(\Omega)$ as $k \rightarrow \infty$. Since continuity implies lower semicontinuity, we have 
		\begin{align} \label{lowerJ}
		\int_\Omega \int_\Omega \J(x-y)(\varphi(x)-\varphi(y))^2\d x \d y =\liminf_{k \rightarrow \infty} \int_\Omega \int_\Omega \J(x-y)(\varphi_k(x)-\varphi_k(y))^2\d x \d y.
		\end{align}
		Combining \eqref{lowerG} and \eqref{lowerJ}, we get
		\begin{align*}
		f(\varphi)\leq \liminf_{k \rightarrow \infty} f(\varphi_k).
		\end{align*}
		This proves $f$ is lower semicontinuous.
	\end{proof}
	
	\begin{remark} \label{lower_H1}
		The proper convexity of $f: \mathrm{V}_0 \rightarrow (-\infty,\infty]$ is immediate, since $\mathrm{V}_0 \hookrightarrow\mathrm{L}^2_{(0)}(\Omega)$. Let $(\varphi_k)_{k \in \R} \in \mathrm{V}_0 $ be  such that $\varphi_k \rightarrow \varphi$ in $\mathrm{V}_0$.	Then from Lemma \ref{poin}, we can easily see that $\varphi_k \rightarrow \varphi $ in $\L^2_{(0)}(\Omega)$. Therefore using Lemma \ref{lowersemi}, we get $f: \mathrm{V}_0 \rightarrow (-\infty,\infty]$ is lower semicontinuous.  
	\end{remark}
	\begin{proposition} \label{maximalmonotone}
		The subgradients $\partial_{\L^2_{(0)}} f$ and $\partial_{\mathrm{V}_0} f$ are maximal monotone operators.
	\end{proposition} 
	\begin{proof}
		In Lemma \ref{lowersemi}, we have proved that  $f:\L^2_{(0)}(\Omega) \rightarrow \mathbb{R}$ is proper convex and lower semicontinuous. By using Lemma \ref{rockfellar}, we obtain  that the operator $\partial_{\L^2_{(0)}} f : \L^2_{(0)}(\Omega) \rightarrow \mathcal{P}((\L^2_{(0)}\Omega)')$ is maximal monotone. Now for the operator $\partial_{\mathrm{V}_0} f$, Remark \ref{lower_H1} gives that $f: \mathrm{V}_0 \rightarrow (-\infty,\infty]$ is proper convex and lower semicontinuous. Hence, once again using Lemma \ref{rockfellar}, we get that $\partial_{\mathrm{V}_0} f$ is also maximal monotone.
	\end{proof}

	\begin{lemma}[Lemma 3.7, \cite{MR3524178}]
		Consider the functional $f$ as in \eqref{functional}. Then for every $\varphi \in \mathrm{D}(\partial_{\L^2_{(0)}}f) \cap \mathrm{V}_0$ we have that $\partial_{\L^2_{(0)}}f(\varphi) \subseteq \partial_{\mathrm{V}_0}f(\varphi)$.
	\end{lemma}
	\begin{lemma}[Lemma 3.8, \cite{MR3524178}] \label{subgrad H1_0}
		Let $\varphi \in \mathrm{D}(\partial_{\mathrm{V}_0} f)$ and $w \in \partial_{\mathrm{V}_0}f(\varphi)$. Suppose $w \in \L^2_{(0)}(\Omega)$, then $\varphi \in \mathrm{D}(\partial_{\L^2_{(0)}}f)$  and 
		\begin{align*}
		w=\partial_{\L^2_{(0)}}f(\varphi) =  a \varphi - \J*\varphi + \mathrm{P}_0 \mathrm{G}'(\varphi).
		\end{align*}
	\end{lemma}
	
	\subsection{Abstract Formulation}
	In this subsection, we define the following spaces in order to set up the problem in Browder's theorem (see Theorem \ref{Browder}).
	\begin{align*}
	&\X :=\V_{\mathrm{div}}\times \mathrm{V}_0 \times \mathrm{V}_0.
	\end{align*}
	Let us define,  
	\begin{align} \label{Zdef}
	\Z:=\left\{ \varphi \in \mathrm{V}_0 :\varphi (x) \in [-1,1],\  \text{ a.e.} \right\},
	\end{align}
	and
	\begin{align*}
	\Y:=\V_{\mathrm{div}}\times \mathrm{V}_0 \times \Z.
	\end{align*}
	Clearly $\Y$ is a closed subspace of $\X$. Let $\mathrm{D}(\T):= \V_{\mathrm{div}}\times \mathrm{V}_0 \times \mathrm{D}(\partial_{\mathrm{V}_0}f)$ and we define a mapping $\T:\Y \rightarrow \mathcal{P}(\X')$ by
	\begin{equation} \label{operator T}
	\T(\mathbf{u},\mu_0, \varphi) := 
	\left\{\begin{array}{cl}
	\left\{ \left( \begin{array}{c}
	0\\
	0\\
	\partial_{\mathrm{V}_0} f(\varphi)    
	\end{array}      \right) \right\}, & \text{if } (\mathbf{u},\mu_0, \varphi) \in \mathrm{D}(\T), \\
	\varnothing, & \text{otherwise.}
	\end{array}\right.
	\end{equation}
	We define the operator $\S: \Y \rightarrow \X'$ as 
	\begin{align} \label{operator S}
	\nonumber \!_{\X'}\langle \S\mathbf{x},\mathbf{y} \rangle_{\X} &:= \int_\Omega (\mathbf{u} \cdot \nabla) \mathbf{u} \cdot \mathbf{v} \d x+ \int_\Omega 2\nu (\varphi) \mathrm{D}\mathbf{u} \cdot \mathrm{D}\mathbf{v} \d x - \int_\Omega\mu_0 \nabla \varphi \cdot \mathbf{v} \d x + \int_\Omega (\mathbf{u} \cdot \nabla \varphi) \cdot \eta \d x \\
	&\quad+\int_\Omega m(\varphi) \nabla \mu_0 \cdot \nabla \eta \d x -\int_\Omega \mu_0 \psi \d x -\int_\Omega \mathrm{P}_0 (\kappa \varphi) \psi \d x,
	\end{align}
	for all $\mathbf{x}=(\mathbf{u},\mu_0, \varphi) \in \Y$, $\mathbf{y}=(\mathbf{v},\eta, \psi) \in \X$ and $\mathbf{b} \in \X'$ is defined by 
	\begin{align*}
	\!_{\X'}\langle \mathbf{b},\mathbf{y} \rangle_{\X} := \langle \mathbf{h} , \mathbf{v} \rangle,
	\end{align*}
	for all $\mathbf{y}\in \X$. From the relations \eqref{operator T} and \eqref{operator S}, the problem $b \in \T(\mathbf{u}, \mu_0,\varphi) +\S(\mathbf{u}, \mu_0,\varphi)$ with $(\mathbf{u}, \mu_0,\varphi) \in \Y\cap \mathrm{D}(\T) $ can be written as
	\begin{equation} \label{operatorform}
	\left( \begin{array}{c}
	0 \\
	0 \\
	\partial_{\mathrm{V}_0} f(\varphi)
	\end{array}
	\right)
	+ 
	\left(  
	\begin{array}{c}
	(\mathbf{u} \cdot \nabla)\mathbf{u} -\text{div}(2\nu(\varphi) \mathrm{D}\mathbf{u})-\mu_0  \nabla \varphi\\
	-\text{div}(m(\varphi) \nabla \mu_0)+\mathbf{u} \cdot \nabla \varphi \\
	-\mu_0-\mathrm{P}_0(\kappa \varphi)
	\end{array}
	\right)
	=
	\left(
	\begin{array}{c}
	\mathbf{h} \\
	0\\
	0
	\end{array}
	\right),
	\end{equation}
	in $\X'$.

	If we can prove that \eqref{operatorform} has a solution, then this solution solves the reformulated problem \eqref{reformphi}-\eqref{nsreform}. This is the content of the next lemma. Later, we discuss the existence of solution for \eqref{operatorform}. 
	
	\begin{lemma} \label{oper_reform}
		Let $(\mathbf{u}, \mu_0,\varphi)\in \Y \cap \mathrm{D}(T)$ satisfies $ \mathbf{b} \in \T(\mathbf{u}, \mu_0,\varphi) + \S(\mathbf{u}, \mu_0,\varphi)$. Then $(\mathbf{u}, \mu_0,\varphi)$ is a solution of the reformulated problem \eqref{reformphi}-\eqref{nsreform}.
	\end{lemma}
	\begin{proof}
		Let $(\mathbf{u},\mu_0,\varphi)$ be such that $ \mathbf{b} \in \T(\mathbf{u}, \mu_0,\varphi) + \S(\mathbf{u}, \mu_0,\varphi)$. From first and second equations of \eqref{operatorform}, it clearly follows that \eqref{reformphi} and \eqref{nsreform} are satisfied for all $\mathbf{v} \in \V_{\mathrm{div}}$ and $\psi \in \mathrm{V}_0$.  Now from the third equation of \eqref{operatorform}, there exists $w \in \partial_{\mathrm{V}_0}f(\varphi)$ such that 
		\begin{align*}
		w = \mu_0 +\mathrm{P}_0(\kappa \varphi)  \ \text{ in  } \ \mathrm{V}_0'.
		\end{align*}
		Since $\mu_0 + \mathrm{P}_0(\kappa \varphi) \in \L^2_{(0)}(\Omega)$ and $\varphi \in \mathrm{D}(\partial_{\mathrm{V}_0}f)$, we can see from Lemma \ref{subgrad H1_0} that $w=a\varphi - \J*\varphi + \mathrm{P}_0(\mathrm{G}'(\varphi))$ in $\mathrm{V}_0'$. This gives for every $\psi \in \mathrm{V}_0$
		\begin{align*}
		\int_\Omega (\mu_0 + \mathrm{P}_0(\kappa \varphi )) \psi \d x  = \int_\Omega ( a\varphi - \J*\varphi ) \psi \d x+ \int_\Omega \mathrm{P}_0(\mathrm{G}'(\varphi))\d x .
		\end{align*}
		Hence 
		\begin{align*}
		\int_\Omega \mu_0 \psi \d x & = \int_\Omega (a\varphi - \J*\varphi ) \d x+ \int_\Omega \mathrm{P}_0(\mathrm{G}'(\varphi)-\kappa \varphi) \d x \\
		&= \int_\Omega (a\varphi - \J*\varphi) \psi \d x  + \int_\Omega \mathrm{P}_0(\mathrm{F}'(\varphi)) \d x.
		\end{align*}
		for all $\psi\in \mathrm{V}_0$.
	\end{proof}
	
	In Lemma \ref{oper_reform}, we showed that the existence of $(\mathbf{u}, \mu_0, \varphi)$ such that $\mathbf{b} \in \T(\mathbf{u}, \mu_0, \varphi)+\S(\mathbf{u}, \mu_0, \varphi)$ implies $(\mathbf{u}, \mu_0, \varphi)$ satisfies the reformulation \eqref{reformphi}-\eqref{nsreform}. Now we show that there exists a $(\mathbf{u}, \mu_0, \varphi) \in \Y \cap \mathrm{D}(\T)$ which satisfies $\mathbf{b} \in \T(\mathbf{u}, \mu_0, \varphi)+\S(\mathbf{u}, \mu_0, \varphi)$. To this purpose, we use Browder's theorem (see Theorem \ref{Browder}).  
	\begin{lemma}\label{browderproof}
		Let $\T,\S$ be as defined in \eqref{operator T} and \eqref{operator S}. Given $\mathbf{b} \in \X'$, there exists a triple $\mathbf{x}=(\mathbf{u}, \mu_0, \varphi) \in \Y \cap \mathrm{D}(\T)$ such that $\mathbf{b} \in \T(\mathbf{x}) + \S(\mathbf{x})$.
	\end{lemma}
	\begin{proof}
		Let us prove that the operators $\T$ and $\S$, and the spaces $\X$ and $\Y$ satisfy the hypothesis of Browder's theorem (see Theorem \ref{Browder}) in the following steps.
		
		\textbf{(1)} We can see that the set $\Y$ is non-empty, closed, convex subset of $\X$. Also $\X$ is a reflexive real Banach space, since $\V_{\mathrm{div}}$ and $\mathrm{V}_0$ are reflexive.  
		
		\textbf{(2)} Now we show that the operator $\T:\Y \rightarrow \mathcal{P}(\X') $ is maximal monotone. Let us first show that $\mathrm{D}(\partial_{\mathrm{V}_0}f) \subseteq \Z. $  In order to get this result, let us take $\varphi \in \mathrm{D}(\partial_{\mathrm{V}_0}f)$.  Then we know that $f(\varphi) \neq \infty$, since $\mathrm{D}(\partial_{\mathrm{V}_0}f) \subseteq \mathrm{D}(f)$. This gives that $\varphi (x) \in [-1,1]$, since $\mathrm{G}(x)=+\infty  $ for $x \notin [-1,1]$. Hence $\varphi \in \Z$ and $\mathrm{D}(\partial_{\mathrm{V}_0}f) \subseteq \Z. $  From Proposition \ref{maximalmonotone}, the operator $ \partial_{\mathrm{V}_0} f: \mathrm{V}_0 \rightarrow \mathcal{P}(\mathrm{V}_0')$  is maximal monotone. This implies that the  operator $\T : \X \rightarrow \mathcal{P}(\X')$ is maximal monotone. Observe that,
		\begin{align*}
		\mathrm{D}(\T)= \V_{\mathrm{div}}\times \mathrm{V}_0 \times \mathrm{D}(\partial_{\mathrm{V}_0}f) \subseteq \V_{\mathrm{div}}\times \mathrm{V}_0 \times \Z = \Y \subset \X.
		\end{align*}
		Moreover by the definition of $\T$, for $(\u, \mu_0, \varphi) \notin \mathrm{D}(\T)$, $\T(\u, \mu_0, \varphi) = \varnothing$. Hence it follows that $\T: \Y \rightarrow \mathcal{P}(\X')$ is a maximal monotone operator.  
		
		\textbf{(3)} We write $\S = \S_1 + \ldots +\S_7$ and  show that each $\S_i$, for $i=1,\ldots,7$, is pseudo-monotone. Let us define 
		\begin{align*}
		& _{\X'}\langle \S_1(\mathbf{u}, \mu_0, \varphi), (\mathbf{v}, \eta , \psi) \rangle_{\X} := \int_\Omega (\mathbf{u} \cdot \nabla ) \mathbf{u} \cdot \mathbf{v} \d x, \\
		& _{\X'}\langle \S_2(\mathbf{u}, \mu_0, \varphi), (\mathbf{v}, \eta , \psi) \rangle_{\X} := \int_\Omega 2\nu (\varphi) \mathrm{D}\mathbf{u} \cdot \mathrm{D}\mathbf{v} \d x,\\
		& _{\X'}\langle \S_3(\mathbf{u}, \mu_0, \varphi), (\mathbf{v}, \eta , \psi) \rangle_{\X} := \int_\Omega \mu_0 \nabla \varphi \cdot \mathbf{v} \d x, \\
		& _{\X'}\langle \S_4(\mathbf{u}, \mu_0, \varphi), (\mathbf{v}, \eta , \psi) \rangle_{\X} := \int_\Omega m(\varphi) \nabla \mu_0 \cdot \nabla \eta \d x, \\
		& _{\X'}\langle \S_5(\mathbf{u}, \mu_0, \varphi), (\mathbf{v}, \eta , \psi) \rangle_{\X} := \int_\Omega (\mathbf{u} \cdot \nabla \varphi )\eta \d x, \\
		& _{\X'}\langle \S_6(\mathbf{u}, \mu_0, \varphi), (\mathbf{v}, \eta , \psi) \rangle_{\X} := \int_\Omega \mu_0 \psi  \d x, \\
		& _{\X'}\langle \S_7(\mathbf{u}, \mu_0, \varphi), (\mathbf{v}, \eta , \psi) \rangle_{\X} := \int_\Omega \mathrm{P}_0(\kappa \varphi ) \psi \d x.
		\end{align*}
		Since completely continuous implies pseudo-monotone, we show that $\S_1, \S_3, \S_5,\S_6$ and $\S_7$ are completely continuous operators. Let us denote $\mathbf{x}_n = (\mathbf{u}_n, \mu_{0n}, \varphi_n)$, $\mathbf{x} = (\mathbf{u}, \mu_0, \varphi)$ and $\mathbf{y}=(\mathbf{v}, \eta , \psi)$. 	Assume that $\mathbf{x}_n \rightharpoonup \mathbf{x}$ in $\Y$. This means $\mathbf{u}_n \rightharpoonup \mathbf{u} $ in $\V_{\mathrm{div}}$ ,  $\mu_{0n} \rightharpoonup \mu_0 $ in $\mathrm{V}_0$ and $\varphi_n \rightharpoonup \varphi $ in $\mathrm{V}_0$, which in turn gives $\mathbf{u}_n \rightarrow \mathbf{u} $ in $\G_{\mathrm{div}}$, $\mu_{0n} \rightarrow \mu_0 $ in $\L^2_{(0)}(\Omega)$ and $\varphi_n \rightarrow \varphi $ in $\L^2_{(0)}(\Omega)$, using the compact embeddings $\V_{\mathrm{div}}\hookrightarrow\G_{\text{div }}$ and $\mathrm{V}_0\hookrightarrow\mathrm{L}^2_{(0)}(\Omega)$.   We have to show that $\S_1 \mathbf{x}_n$ converges strongly to $\S_1 \mathbf{x} $ in $\X'$-norm.	Using \eqref{2.7}, H\"older's and Ladyzhenskaya's inequalities, for $n=2$, we get 
		\begin{align}
		|\,_{\X'}\langle \S_1 \mathbf{x}_n, \mathbf{y} \rangle_{\X}-\,_{\X'}\langle \S_1 \mathbf{x}, \mathbf{y} \rangle_{\X}| \nonumber=& |b(\mathbf{u}_n, \mathbf{u}_n,\mathbf{v})- b(\mathbf{u},\mathbf{u},\mathbf{v})| \nonumber \\
		\nonumber=& |-b(\mathbf{u}_n, \mathbf{v}, \mathbf{u}_n -\u) -b(\mathbf{u}_n-\mathbf{u},\mathbf{v},\mathbf{u})| \\
		\nonumber\leq & \|\mathbf{u}_n \|_{\mathbb{L}^4} \| \nabla \mathbf{v}\| \ \|\mathbf{u}_n-\mathbf{u} \|_{\mathbb{L}^4} + \|\mathbf{u} \|_{\mathbb{L}^4} \| \nabla \mathbf{v}\| \ \|\mathbf{u}_n-\mathbf{u} \|_{\mathbb{L}^4} \\ 
		\label{S12} \leq & 2^{1/4}  \left(\|\mathbf{u}_n\|^{1/2} \| \nabla \mathbf{u}_n\|^{1/2}   +   \|\mathbf{u}\|^{1/2} \| \nabla \mathbf{u}\|^{1/2} \right) \|\mathbf{u}_n - \mathbf{u} \|_{\mathbb{L}^4}  \| \mathbf{v}\|_{\V_{\mathrm{div}}}.  \qquad
		\end{align}
		For $n=3$, using H\"older's and Ladyzhenskaya's inequalities,, we get 
		\begin{align} \label{S13}
		|\,_{\X'}\langle \S_1 \mathbf{x}_n, \mathbf{y} \rangle_{\X} -\,_{\X'}\langle \S_1 \mathbf{x}, \mathbf{y} \rangle_{\X}| \nonumber \nonumber
		\leq & \| \mathbf{u}_n \|_{\mathbb{L}^4} \| \nabla \mathbf{v} \| \ \| \mathbf{u}_n - \mathbf{u} \|_{\mathbb{L}^4} + \| \mathbf{u}_n - \mathbf{u} \|_{\mathbb{L}^4} \| \nabla \mathbf{v} \| \|\mathbf{u} \|_{\mathbb{L}^4} \\ \leq & 2^{1/2}\left(\|\u_n\|^{1/4}\|\nabla\u_n\|^{3/4}+\|\u\|^{1/4}\|\nabla\u\|^{3/4}\right)\|\u_n-\u\|_{\mathbb{L}^4}\|\v\|_{\V_{\mathrm{div}}}.
		\end{align}
		Let us now estimate $|\,_{\X'}\langle \S_3\mathbf{x}_n,\mathbf{y} \rangle - \langle \S_3 \mathbf{x}, \mathbf{y} \rangle_{\X} |$ and $|\,_{\X'}\langle \S_5\mathbf{x}_n,\mathbf{y} \rangle_{\X}- \,_{\X'}\langle \S_5 \mathbf{x}, \mathbf{y} \rangle_{\X}|$. We perform an integration by parts, use H\"older's inequality and the embedding $\H_0^1(\Omega)\hookrightarrow\mathbb{L}^4(\Omega)$ to estimate $|\,_{\X'}\langle \S_3\mathbf{x}_n,\mathbf{y} \rangle - \langle \S_3 \mathbf{x}, \mathbf{y} \rangle_{\X} |$ as 
		\begin{align} \nonumber
		|\,_{\X'}\langle \S_3\mathbf{x}_n,\mathbf{y} \rangle - \langle \S_3 \mathbf{x}, \mathbf{y} \rangle_{\X} |&= \left|-\int_\Omega \mu_{0n} \nabla \varphi_n \cdot \mathbf{v} \d x +\int_\Omega \mu_0 \nabla \varphi \cdot \mathbf{v} \d x\right| \\ \nonumber
		&=\left|\int_\Omega \mu_{0n} \nabla (\varphi - \varphi_n) \cdot \mathbf{v} \d x + \int_\Omega (\mu_0-\mu_{0n})\nabla \varphi \cdot \mathbf{v} \d x  \right|\\ \nonumber
		& \leq  \| \nabla \mu_{0n}\| \ \| \mathbf{v}\|_{\mathbb{L}^4} \|\varphi - \varphi_n \|_{\L^4} + \|\mu_0-\mu_{0n} \|_{\L^4} \|\nabla \varphi \| \ \|\mathbf{v} \|_{\mathbb{L}^4} \\ \label{S3}
		& \leq C\left(  \| \nabla \mu_{0n}\| \|\varphi - \varphi_n \|_{\L^4} +  \|\mu_0-\mu_{0n} \|_{\L^4} \|\nabla \varphi \|\right)  \| \mathbf{v} \|_{\V_{\mathrm{div}}}.
		\end{align}
		Similarly, we have 
		\begin{align} \nonumber
		|\,_{\X'}\langle \S_5\mathbf{x}_n,\mathbf{y} \rangle_{\X}- \,_{\X'}\langle \S_5 \mathbf{x}, \mathbf{y} \rangle_{\X}| &= \left|\int_\Omega (\mathbf{u}_n \cdot \nabla \varphi_n) \eta \d x -\int_\Omega (\mathbf{u} \cdot \nabla \varphi) \eta \d x \right|\\ \nonumber
		&= \left|\int_\Omega (\mathbf{u}_n-\mathbf{u}) \cdot \nabla \varphi_n \eta \d x+\int_\Omega \mathbf{u} \cdot \nabla (\varphi_n -\varphi)\eta \d x\right| \\ \label{S5}
		&\leq \left( \|\mathbf{u}_n-\mathbf{u} \|_{\mathbb{L}^4}\|  \varphi_n\|_{\L^4} +\| \mathbf{u}\|_{\mathbb{L}^4} \ \|\varphi_n -\varphi \| _{\L^4} \right)\|\nabla \eta \|.
		\end{align}
		Using H\"older's inequality,  we obtain
		\begin{align}\label{S6}
		|\,_{\X'}\langle \S_6 \mathbf{x}_n, \mathbf{y} \rangle_{\X}-\,_{\X'}\langle \S_6 \mathbf{x}, \mathbf{y} \rangle_{\X}| \leq & \int_\Omega |\mu_{0n} - \mu_0|\,|\psi| \d x 
		\leq  \|\mu_{0n} - \mu_0 \| \, \|\psi \|
		\end{align}
		and
		\begin{align}\label{S7}
		|\,_{\X'}\langle \S_7 \mathbf{x}_n, \mathbf{y} \rangle_{\X} -\,_{\X'}\langle \S_7 \mathbf{x}, \mathbf{y} \rangle_{\X}| &= \left|\int_\Omega\mathrm{P}_0(\kappa \varphi_n) \psi \d x - \int_\Omega \mathrm{P}_0(\kappa \varphi) \psi \d x \right| 
		\leq  \kappa \int_{\Omega} |\varphi_n - \varphi |  |\psi| \d x \nonumber\\&
		\leq  \kappa \|\varphi_n -\varphi \| \|\psi \|. 
		\end{align}
		From \eqref{S12}-\eqref{S13}, we get
		\begin{align*}
		&	\| \S_1 \mathbf{x}_n -\S_1\mathbf{x}\|_{\X'} \\ &\leq 2^{1/2}  \left(\|\mathbf{u}_n\|^{1/2} \| \nabla \mathbf{u}_n\|^{1/2}   +   \|\mathbf{u}\|^{1/2} \| \nabla \mathbf{u}\|^{1/2} \right)( \|\mathbf{u}_n \|^{1/2}+\|\mathbf{u} \|^{1/2}) \|\mathbf{u}_n - \mathbf{u} \|^{1/2}\ \ (n=2),
		\end{align*} 
		and 
		\begin{align*}
		&	\| \S_1 \mathbf{x}_n -\S_1\mathbf{x}\|_{\X'}  \\ &\leq 2\left(\|\u_n\|^{1/4}\|\nabla\u_n\|^{3/4}+\|\u\|^{1/4}\|\nabla\u\|^{3/4}\right)( \|\mathbf{u}_n \|^{3/4}+\|\mathbf{u} \|^{3/4})\|\u_n-\u\|^{1/4} \ \ (n=3),
		\end{align*} 
		and both converges to $0$ as $n\to\infty$ using the compact embedding of $\V_{\mathrm{div}}\hookrightarrow\G_{\mathrm{div}}$. Using Ladyzhenskaya's inequality and the compact embedding $\mathrm{V}_0 \hookrightarrow \L^2_{(0)}(\Omega)$,  from \eqref{S3}-\eqref{S7}, we have 
		\begin{align*}
		\|  \S_3\mathbf{x}_n- \S_3\mathbf{x}\|_{\X'}  &\leq  \| \nabla \mu_{0n}\| \|\varphi - \varphi_n \|_{\L^4} +  \|\mu_0-\mu_{0n} \|_{\L^4} \|\nabla \varphi \|  \rightarrow 0, \\
		\| \S_5\mathbf{x}_n- \S_5\mathbf{x} \|_{\X'}  &\leq \|\mathbf{u}_n-\mathbf{u} \|_{\mathbb{L}^4}\|  \varphi_n\|_{\L^4} +\| \mathbf{u}\|_{\mathbb{L}^4} \ \|\varphi_n -\varphi \| _{\L^4} \rightarrow 0, \\
		\|\S_6 \mathbf{x}_n- \S_6 \mathbf{x} \|_{\X'} & \leq \| \mu_{0n}- \mu_0\| \rightarrow 0, \\
		\|\S_7 \mathbf{x}_n- \S_7 \mathbf{x} \|_{\X'}  &\leq \kappa \| \varphi_n - \varphi \| \rightarrow 0,
		\end{align*}
		as $n\to\infty$. This proves that $\S_1, \S_3, \S_5, \S_6$ and $\S_7$ are completely continuous and hence pseudo-monotone.
		In order to prove the pseudo-monotonicity of $\S_4$, we use Lemma \ref{pseudo-monotone}. Let us define the operator $\widetilde{\S}_4 : \X  \times \X \rightarrow \X'$ by 
		\begin{align*}
		\,_{\X'}\langle \widetilde{\S}_4  (\mathbf{x}_1,\mathbf{x}_2),\mathbf{y} \rangle_{\X} := \int_\Omega m (\varphi_1) \nabla \mu_{0_2} \cdot \nabla \eta\d x,
		\end{align*}
		where $\mathbf{x}_i = (\mathbf{u}_i, \mu_{0_i},\varphi_i) \in \X, i=1,2,$ respectively and $\mathbf{y}=(\mathbf{v},\eta,\psi) \in \X.$ If $\x=(\u,\mu_0,\varphi)\in\X$, then 
		\begin{align*}
		\,_{\X'}\langle \widetilde{\S}_4 (\mathbf{x},\mathbf{x}_1 - \mathbf{x}_2) , \mathbf{x}_1 - \mathbf{x}_2 \rangle_{\X}  &= \int_\Omega m(\varphi_1) \nabla (\mu_{0_1}- \mu_{0_2}) \cdot \nabla (\mu_{0_1}- \mu_{0_2})\d x\\&=\int_\Omega m(\varphi_1) |\nabla (\mu_{0_1}- \mu_{0_2}) |^2\d x \geq 0,
		\end{align*}
		since $m(\cdot)\geq 0$. This implies $\widetilde{\S}_4(\mathbf{x}, \cdot)$ is monotone. Now for each fixed $\x\in\X$, it can be easily seen that $\,_{\X'}\langle \widetilde{\S}_4(\x,\x_1+t_k\x_2),\y\rangle_{\X}\to \,_{\X'}\langle \widetilde{\S}_4(\x,\x_1),\y\rangle_{\X}$, as $t_k\to0$. Hence, the operator $\widetilde{\S}_4(\mathbf{x}, \cdot)$ is hemicontinuous for each fixed $\x\in\X$.  In order to prove complete continuity of $\widetilde{\S}_4(\cdot ,\mathbf{x})$ for all $\mathbf{x} \in \X$, we consider a sequence $\widetilde{\x}_k \subseteq \X$ such that $\widetilde{\x}_k \rightharpoonup \widetilde{\x}$ in $\X$ and $\widetilde{\x}=(\widetilde{\mathbf{u}}, \widetilde{\mu}_0, \widetilde{\varphi}) \in \X$. Then 
		\begin{align*}
		\,_{\X'}\langle \widetilde{\S}_4(\widetilde{\x}_k ,\mathbf{x}) - \widetilde{\S}_4(\widetilde{\x} ,\mathbf{x}) , \mathbf{y}\rangle_{\X} &= \int_\Omega (m(\widetilde{\varphi}_k) \nabla \mu_0 - m(\widetilde{\varphi}) \nabla \mu_0) \cdot \nabla \mathbf{v}\d x \\
		&\leq \| m(\widetilde{\varphi}_k) \nabla \mu_0 - m(\widetilde{\varphi}) \nabla \mu_0  \| \ \|\mathbf{v} \|_{\mathbb{V}_{\mathrm{div}}}.
		\end{align*}
		We know that $\widetilde{\varphi}_k \rightarrow \widetilde{\varphi}$ in $\L_{(0)}^2(\Omega)$ and $m$ is a continuous function, we have  to show that $\| m(\widetilde{\varphi}_k) \nabla \mu_0 - m(\widetilde{\varphi}) \nabla \mu_0  \| \rightarrow 0$ as $k\to\infty$. Let us define 
		$$\mathrm{H}(\widetilde{\varphi})(x):=g(x,\widetilde{\varphi}(x))=m(\widetilde{\varphi}(x))\nabla \mu_0(x).$$ Then Lemma 1.19, \cite{ruzicka2006nichtlineare} yields that  $\mathrm{H}:\mathrm{L}^2_{(0)}(\Omega)\to\mathrm{L}^2(\Omega)$ is continuous and bounded. Since $\widetilde{\varphi}_k \rightarrow \widetilde{\varphi}$ in $\L_{(0)}^2(\Omega)$, this implies that  $\| m(\widetilde{\varphi}_k) \nabla \mu_0 - m(\widetilde{\varphi}) \nabla \mu_0  \| \rightarrow 0$ as $k\to\infty$ and hence $\widetilde{\S}_4(\cdot,\x)$ is completely continuous for all $\x\in\X$.
		Using Lemma \ref{pseudo-monotone}, we get that the operator  $\S_4(\x)=\widetilde{\S}_4(\x,\x)$ is pseudo-monotone . A similar proof follows for $\S_2$, since it is of the same form as $\S_4$ and $\nu(\cdot)$ is a continuously differentiable function.
		
		Since each of the operators $\S_i$'s are bounded, the operator $\S$ is also bounded. Since pseudo-monotonicity and locally boundedness implies demicontinuity (see Lemma 2.4, \cite{MR3014456} or Proposition 27.7, \cite{MR1033498}), $\S$ is demicontinuous.

		\textbf{(4)} Observe that $\Y$ is unbounded. Choose $\x_0 =(\mathbf{0},0,0) \in \Y \cap \mathrm{D}(\T).$ Then we have to show that there is $R>0$ such that 
		\begin{align*}
		\,_{\X'}\langle \S(\mathbf{x})-\mathbf{b}, \mathbf{x} \rangle_{\X} >0, \ \text{ for all } \ \mathbf{x} \in \Y \ \text{ with } \  \|\mathbf{x}\|_{\X} >R.
		\end{align*} 
		For $\x=(\u,\mu_0,\varphi)\in\Y$, let us consider 
		\begin{align}\label{com}
		_{\X'}\langle \S(\mathbf{x})-\mathbf{b}, \mathbf{x} \rangle_{\X} \nonumber&= \int_\Omega2 \nu(\varphi)\mathrm{D}\mathbf{u} \cdot \mathrm{D}\mathbf{u} \d x + \int_\Omega m(\varphi) \nabla \mu_0 \cdot \nabla \mu_0 \d x- \int_\Omega \mu_0 \varphi \d x \\
		\nonumber&\quad- \int_\Omega \mathrm{P}_0(\kappa \varphi) \varphi \d x - \langle \mathbf{h} , \mathbf{u} \rangle\\
		&=: \I_1+ \I_2+ \I_3 + \I_4 +\I_5.
		\end{align}
		Since $|\varphi(x)| \leq 1$, and $\nu(\cdot)$ is a continuously differentiable function, we have 
		\begin{align*}
		\I_1 = \int_\Omega 2\nu(\varphi) \mathrm{D}\mathbf{u} \cdot\mathrm{D}\mathbf{u}  \d x\geq 2\int_{\Omega}\nu(\varphi)|\nabla\u|^2\d x \geq \widetilde{C}_1 \|\mathbf{u} \| ^2 _{\V_{\mathrm{div}}}.
		\end{align*}
		Now, since mean value of $\mu_0$ is $0$ and $\mu(\cdot)$ is a continuous function, we also obtain 
		\begin{align*}
		\I_2=\int_\Omega m(\varphi) \nabla \mu_0 \cdot \nabla \mu_0 \d x=\int_\Omega m(\varphi) |\nabla \mu_0|^2 \d x \geq \widetilde{C}_2 \| \mu_0\|^2_{\mathrm{V}_0}.
		\end{align*}
		Using H\"older's, Poincar\'e's and Young's inequalities, we get 
		\begin{align*}
		\I_3&=\int_\Omega \mu_0 \varphi \d x \leq \int_{\Omega}|\mu_0|\d x\leq |\Omega|^{1/2}\|\mu_0\|_{\mathrm{L}^2}\leq  \widetilde{C}_3 \| \mu_0 \|_{\mathrm{V}_0},\\
		\I_4&=\int_\Omega \mathrm{P}_0(\kappa \varphi) \varphi \d x = \kappa \| \varphi \|^2 \leq \widetilde{C}_4,\\
		\I_5&= \langle \mathbf{h} , \mathbf{u} \rangle \leq \|\h\|_{\V_{\mathrm{div}}'}\|\u\|_{\V_{\mathrm{div}}} \leq \frac{1}{2\widetilde{C}_1} \| \mathbf{h}\|_{\V'_{\mathrm{div}}}^2 +\frac{ \widetilde{C}_1}{2} \|\mathbf{u} \|^2_{\V_{\mathrm{div}}},
		\end{align*}
		for some constants $\widetilde{C}_1, \widetilde{C}_2, \widetilde{C}_3, \widetilde{C}_4 >0$. Combining the above inequalities and substituting in \eqref{com} yields  
		\begin{align*}
		_{\X'}\langle \S(\mathbf{x})-\mathbf{b}, \mathbf{x} \rangle_{\X} & \geq  \frac{\widetilde{C}_1}{2} \|\mathbf{u} \| ^2_{\V_{\mathrm{div}}} + \underbrace{\| \mu_0 \|_{\mathrm{V}_0}\left(\widetilde{C}_2 \| \mu_0 \|_{\mathrm{V}_0} -\widetilde{C}_3\right)}_{:=g\left(\|\mu_0\|_{\mathrm{V}_0}\right)}-\widetilde{C}_4 - \frac{1}{2\widetilde{C}_1} \| \mathbf{h}\|_{\V_{\mathrm{div}}'}^2 
		\end{align*}
		for $\x=(\mathbf{u}, \mu_0, \varphi) \in \Y$. Since $\h\in\G$ is fixed, we can choose a constant $K>0$ large enough such that 
		\begin{align*}
		\widetilde{C}_4 + \frac{1}{2\widetilde{C}_1}\| \mathbf{h}\|_{\V'_{\mathrm{div}}}^2 \leq K. 
		\end{align*}Furthermore, it can be easily seen that 
		$$\lim_{x\to+\infty}g(x)=+\infty.$$
		Since $K$ is chosen and $g (\cdot)$ grows to infinity, we can choose $R>0$  such that
		\begin{align*}
		_{\X'}\langle \S(\mathbf{x})-\mathbf{b}, \mathbf{x} \rangle_{\X }>0 \ \text{ for all } \ \x=(\u,\mu_0,\varphi)\in\Y \ \text{ such that } \ \|\mathbf{x}\|_{\X} >R. 
		\end{align*}
		Hence using Theorem \ref{Browder}, there exists $(\mathbf{u}, \mu_0,\varphi) \in \Y \cap \mathrm{D}(\T)$ such that $\mathbf{b} \in \T(\mathbf{u}, \mu_0,\varphi) +\S(\mathbf{u}, \mu_0,\varphi)$, which completes the proof.
	\end{proof}
	\begin{theorem} \label{mainthm}
		Let $\Omega \subset \mathbb{R}^n$, $n=2,3$ be a bounded subset. Let $\mathbf{h} \in \V'_{{\mathrm{div}}}$ and $k \in (-1,1),$ then under the Assumptions (A1)-(A9),  there exists a weak solution $(\mathbf{u}, \mu,\varphi )$ to the system \eqref{steadysys} such that 
		\begin{align*}
		\mathbf{u} \in \V_{{\mathrm{div}}}, \quad  \quad \varphi \in  \mathrm{V} \cap \L^2_{(k)}(\Omega) \quad \mathrm{and} \quad \mu \in \mathrm{V}.
		\end{align*}	 
		and satisfies the weak formulation \eqref{weakphi}-\eqref{weakform nse}.
	\end{theorem}
	\begin{proof}
		Lemma \ref{browderproof} gives the existence of $(\mathbf{u}, \mu_0,\varphi) \in \Y \cap \mathrm{D}(\T)$ such that $\mathbf{b} \in \T(\mathbf{u}, \mu_0,\varphi) +\S(\mathbf{u}, \mu_0,\varphi)$. From Lemma \ref{oper_reform} we get that $(\mathbf{u}, \mu_0,\varphi)$ satisfies the reformulated problem \eqref{reformphi}-\eqref{nsreform}. Hence from Lemma \ref{reformtoweak}, we know that  $(\mathbf{u}, \mu,\varphi)$ satisfies the weak formulation \eqref{weakphi}-\eqref{weakform nse}, which completes the proof.
	\end{proof} 
		\begin{remark}\label{F'phi}
			Note that, since $\mu \in \mathrm{V}$ and $\varphi \in \mathrm{V}\cap \L^2_{(k)}(\Omega)$ and $\|a\|_{\mathrm{L}^\infty} \leq \|\J\|_{\mathrm{L}^1}$, from \eqref{weakmu}, we have 
			\begin{align*}
			\left| \int_\Omega\F'(\varphi)\psi \d x\right|= \left| \int_\Omega\mu \psi \d x- \int_\Omega(a \varphi-\J * \varphi)\psi \d x \right|\leq (\|\mu \| + \|a\|_{\L^{\infty}} \|\varphi\| + \|\J\|_{\L^1} \|\varphi\| ) \|\psi\|,
			\end{align*}
			for all $\psi \in \mathrm{H}$, so that $\|\F'(\varphi)\| \leq \|\mu\| + 2 \|\J\|_{\mathrm{L}^1}\|\varphi\| < \infty$.
		\end{remark}

	\section{Uniqueness of  Weak Solutions and Regularity}\setcounter{equation}{0}\label{sec4} In this section, we prove that the weak solution to the system \eqref{steadysys} obtained in Theorem \ref{mainthm} is unique. In this section, we assume  that the viscosity coefficient $\nu$ and the mobility parameter $m$ are positive constants. We also prove some regularity results for the solution. 
	\subsection{Uniqueness} Recall that our existence result ensures that $\varphi(\cdot) \in [-1, 1]$, a.e. (see  \eqref{Zdef}). Then, we have the following uniqueness theorem.
	\begin{theorem}\label{unique-steady}
		Let $(\u_i,\mu_i,\varphi_i) \in \V_{{\mathrm{div}}}  \times \mathrm{V} \times (\mathrm{V}\cap \L^2_{(k)}(\Omega))$ for $i=1,2$ be two \emph{weak solutions} of the following system (satisfied in the weak sense).
		\begin{equation}
		\left\{
		\begin{aligned}\label{3.2}
		\mathbf{u} \cdot \nabla \varphi &= m\Delta \mu, \ \emph{ in }\ \Omega, \\
		\mu &= a \varphi - \mathrm{J}*\varphi + \mathrm{F}'(\varphi)  \\
		-\nu \Delta \mathbf{u} + (\mathbf{u} \cdot \nabla) \mathbf{u} + \nabla \uppi &= \mu \nabla \varphi + \mathbf{h}, \ \emph{ in }\ \Omega,\\ 
		\emph{div }\mathbf{u} &= 0, \ \emph{ in }\ \Omega, \\
		\frac{\partial \mu}{ \partial \mathbf{n}}=0, \ \mathbf{u}&=0, \ \emph{ on }\ \partial\Omega,
		\end{aligned}
		\right.
		\end{equation} 
		where $\h\in\V_{{\mathrm{div}}}'$. For $n=2$, we have $\u_1=\u_2$ and $\varphi_1=\varphi_2$, provided  
		\begin{align*}
		& (i) \  \nu^2 >  \frac{2\sqrt{2}}{\sqrt{\lambda_1}} \|\h\|_{\mathbb{V}'_{\mathrm{div}}}  + \frac{12 \nu }{ \lambda_1m C_0} \|\nabla \J \|_{\mathbb{L}^1}^2,\\
		&(ii) \ (\nu m)^2 \left( \frac{C_0}{4} -\frac{C}{C_0} \|\nabla \J\|^2_{\mathbb{L}^1}  \right)>  \nu m  \left(\frac{C}{\lambda_1}\right) + \frac{2 C }{ C_0} \left( \frac{2}{\lambda_1} \right)^{\frac{1}{2}}  \| \h \|^2 _{\mathbb{V}'_{\mathrm{div}}} \ \text{ with }\ \|\nabla\mathrm{J}\|^2_{\mathbb{L}^1}<\frac{C^2_0}{4C},
		\end{align*}
		Similarly, for $n=3$, we have $\u_1=\u_2$ and $\varphi_1=\varphi_2$, provided 
		\begin{align*}
		&(i)\ \nu^2  >  \left( \frac{16}{\sqrt{\lambda_1}} \right)^{\frac{1}{2}}  \|\h\|_{\mathbb{V}'_{\mathrm{div}}} +\frac{12 \nu }{ \lambda_1m C_0} \|\nabla \J \|_{\mathbb{L}^1}^2,  \\
		&(ii)\  (\nu m)^2 \left( \frac{C_0}{4} -\frac{C}{C_0} \|\nabla \J\|_{\mathbb{L}^1}^2 \right)>   \nu m  \left(\frac{C}{\lambda_1}\right) + \frac{2 C }{ C_0} \left( \frac{4}{\sqrt{\lambda_1}} \right)^{\frac{1}{2}}  \| \h \|^2 _{\mathbb{V}'_{\mathrm{div}}},  \ \text{ with }\ \|\nabla\mathrm{J}\|^2_{\mathbb{L}^1}<\frac{C^2_0}{4C},
		\end{align*}	
		where $C$ is a generic constant depends on the embeddings.
	\end{theorem}
	
	\begin{proof}
		Let us first find a simple bound for the velocity field. We take inner product of the first equation in \eqref{3.2} with $\mu$ and third equation with $\u$ to obtain 
		\begin{align} \label{phi}
		(\u \cdot \nabla \varphi, \mu)= -m \| \nabla \mu \|^2
		\end{align}
		and
		\begin{align} \label{u}
		\nu \| \nabla \u\|^2 = (\mu \nabla \varphi, \u) +\langle\h, \u\rangle.
		\end{align}
		Adding \eqref{phi} and \eqref{u}, we obtain 
		\begin{align}\label{1p}
		m\|\nabla\mu\|^2+\nu\|\nabla\u\|^2=\langle\h, \u\rangle,
		\end{align}
		where we used the fact that $(\u \cdot \nabla \varphi, \mu)=(\mu \nabla \varphi, \u)$. From \eqref{1p}, we infer that 
		\begin{align*}
		\nu\|\nabla\u\|^2 \leq m\|\nabla\mu\|^2 + \nu\|\nabla\u\|^2 =\langle\mathbf{h},\u\rangle \leq \| \h \|_{\mathbb{V}_{\mathrm{div}}'} \|\u \|_{\mathbb{V}_{\mathrm{div}}}.
		\end{align*}
		Finally, we have 
		\begin{align} \label{gradu estimate}
		\nu \|\nabla\u\| \leq \| \h \|_{\mathbb{V}_{\mathrm{div}}'}.
		\end{align}
		
		Let $(\u_1, \mu_1,\varphi_1) $ and $(\u_2,\mu_2, \varphi_2)$  be two weak solutions of the system \eqref{3.2}. Note that the averages of $\varphi_1$ and $\varphi_2$ are same and are equal to $k$, which gives $\overline{\varphi}_1 - \overline{\varphi}_2=0$. 
		We can rewrite the third equation in \eqref{3.2} as
		\begin{equation}\label{nonlin u rewritten}
		-\nu \Delta \u + (\u\cdot \nabla )\u + \nabla \widetilde{\uppi}_{\u} = - \nabla a\frac{{\varphi}^2}{2} - (\J\ast \varphi)\nabla \varphi + \h \ \text{ in }\ \V_{{\mathrm{div}}}',
		\end{equation} 
		where $\widetilde{\uppi}_{\u} :=\widetilde{\uppi} = \uppi -\left( \F(\varphi) + a\frac{{\varphi}^2}{2}\right)$. 	Let us define $\u^e:=\u_1-\u_2$, $\varphi^e:=\varphi_1-\varphi_2$ and $\widetilde{\uppi}^e:=\widetilde{\uppi}_{\u_1}-\widetilde{\uppi}_{\u_2}$. Then  for every $\mathbf{v} \in \V_{\mathrm{div}}$ and $\psi \in \mathrm{V}, \, (\u^e,\varphi^e)$ satisfies the following:
		\begin{align*}
		(\mathbf{u}^e \cdot \nabla \varphi_1, \psi)+(\u_2\cdot\nabla\varphi^e,\psi )&= -m(\nabla \mu^e , \nabla \psi), \\
		\nu (\nabla \mathbf{u}^e, \nabla \mathbf{v} )+ b(\mathbf{u}_1,\mathbf{u}^e,\mathbf{v})+ b(\mathbf{u}^e,\mathbf{u}_2,\mathbf{v}) 
		&  = -\frac{1}{2}\left(\varphi^e( \varphi_1+\varphi_2)\nabla a , \mathbf{v}\right)- ((\mathrm{J}*\varphi^e)\nabla\varphi_2 , \mathbf{v})\nonumber\\
		&\quad- ((\mathrm{J}*\varphi_1)\nabla\varphi^e,\mathbf{v}),
		\end{align*}
		Let us choose $\mathbf{v} = \mathbf{u}^e$ and $\psi = \mathcal{B}^{-1} \varphi^e$ to get 
		\begin{align} \label{s48}
		(\mathbf{u}^e \cdot \nabla \varphi_1, \mathcal{B}^{-1} \varphi^e)+(\u_2\cdot\nabla\varphi^e, \mathcal{B}^{-1} \varphi^e) &= m\langle \Delta \mu^e ,\mathcal{B}^{-1} \varphi^e \rangle, \\ \label{s49}
		\nu\|\nabla\u^e\|^2&= -b(\mathbf{u}^e,\mathbf{u}_2,\u^e)-\frac{1}{2}(\nabla a\varphi^e( \varphi_1+\varphi_2),\u^e)\nonumber\\&\quad -( (\mathrm{J}*\varphi^e)\nabla\varphi_2,\u^e) - ((\mathrm{J}*\varphi_1)\nabla\varphi^e,\u^e),
		\end{align}
		where we used the fact that $b(\mathbf{u}_1, \mathbf{u}^e,\u^e)=0$.
		Using \eqref{bes}, Taylor's series expansion and (A9),  we estimate the  term $(\nabla \mu^e ,\nabla \mathcal{B}^{-1} \varphi^e)$ from \eqref{s48} as  
		\begin{align}\label{3.13}
		-\langle-\Delta \mu^e,\mathcal{B}^{-1} \varphi^e\rangle & =-(\mu^e, \varphi^e )
		=	-(a \varphi^e - \mathrm{J}*\varphi^e + \mathrm{F}'(\varphi_1)-\mathrm{F}'(\varphi_2),\varphi^e ) \nonumber \\ 
		& = -((a  + \mathrm{F}''(\varphi_1+ \theta \varphi_2)) \varphi^e, \varphi^e ) + (\mathrm{J}*\varphi^e, \varphi^e) \nonumber\\
		& \leq -C_0 \| \varphi^e \|^2 + (\mathrm{J}*\varphi^e, \varphi^e).
		\end{align}
		Using \eqref{3.13} in  \eqref{s48}, we obtain 
		\begin{align}\label{3.14}
		mC_0\|\varphi^e\|^2 \leq   -(\u^e\cdot\nabla\varphi_1,\mathcal{B}^{-1}\varphi^e)-(\u_2\cdot\nabla\varphi^e,\mathcal{B}^{-1}\varphi^e) + m(\J\ast \varphi^e , \varphi^e).
		\end{align}
		
		Let us now estimate the terms in the right hand side of \eqref{s49} and \eqref{3.14} one by one. We use the H\"older's, Ladyzhenskaya's, Young's and Poincar\'{e}'s inequalities to estimate $|b(\mathbf{u}^e,\mathbf{u}_2,\u^e)|$. 
		For $n=2$, we get
		\begin{equation}\label{3.15a}
		|b(\mathbf{u}^e,\mathbf{u}_2,\u^e)|\leq \|\u^e\|_{\mathbb{L}^4}^2\|\nabla\u_2\|\leq \sqrt{2}\|\u^e\|\|\nabla\u^e\|\|\nabla\u_2\|\leq \left( \frac{2}{\lambda_1} \right)^{\frac{1}{2}} \|\nabla\u^e\|^2 \|\nabla\u_2\|,
		\end{equation}
		and	for $n=3$, we have 
		\begin{equation}\label{3.15b}
		|b(\mathbf{u}^e,\mathbf{u}_2,\u^e)|\leq \|\u^e\|_{\mathbb{L}^4}                                         ^2\|\nabla\u_2\| \leq 2\|\u^e\|^{\frac{1}{2}} \|\nabla\u^e\|^{\frac{3}{2}} \|\nabla\u_2\| \leq \left( \frac{4}{\sqrt{\lambda_1}} \right)^{\frac{1}{2}} \|\nabla\u^e\|^2 \|\nabla\u_2\|.
		\end{equation}
		Since $(\u_i, \varphi_i)$ are weak solutions of \eqref{3.2} for $i=1,2,$ such that $\varphi_i \in [-1, 1] $, a.e., we have,  $ |\varphi_i (x) | \leq 1$.
		Using H\"older's, Ladyzhenskaya's and Young's inequalities and boundedness of $\varphi$, we obtain 	
		\begin{align}\label{3.16a}
		\left|\frac{1}{2}(\nabla a\varphi^e( \varphi_1+\varphi_2),\u^e)\right| &\leq \frac{1}{2}\|\nabla a\|_{\mathbb{L}^{\infty}}\|\varphi^e\|\|\varphi_1+\varphi_2\|_{\mathrm{L}^{\infty}}\|\u^e\| \no\\
		&\leq \frac{1}{2}\|\nabla a\|_{\mathbb{L}^{\infty}}\|\varphi^e\|(\|\varphi_1\|_{\mathrm{L}^{\infty}} +\|\varphi_2\|_{\mathrm{L}^{\infty}}) \|\u^e\| \no\\
		&\leq \frac{m C_0}{8}\|\varphi^e\|^2+\frac{2}{m C_0}\|\nabla a\|_{\mathbb{L}^{\infty}}^2\|\u^e\|^2.
		\end{align}
		In order to estimate $((\mathrm{J}*\varphi^e)\nabla\varphi_2,\u^e)$ and $((\mathrm{J}*\varphi_1)\nabla\varphi^e,\u^e)$, we write these terms using an integration by parts and the divergence free condition as
		\begin{align*}
		((\mathrm{J}*\varphi^e)\nabla\varphi_2,\u^e) = -((\nabla\mathrm{J}*\varphi^e)\varphi_2,\u^e), \\
		((\mathrm{J}*\varphi_1)\nabla\varphi^e,\u^e) = -((\nabla\mathrm{J}*\varphi_1)\varphi^e,\u^e).
		\end{align*}
		Using H\"older's and Ladyzhenskaya's inequalities,  and Young's inequality for convolution, we estimate $|((\nabla\mathrm{J}*\varphi^e)\varphi_2,\u^e)|$  as
		\begin{align}\label{3.17a}
		|((\nabla\mathrm{J}*\varphi^e)\varphi_2,\u^e)| &\leq \| \nabla \J*\varphi^e\|_{\mathbb{L}^2}\|\varphi_2\|_{\mathrm{L}^\infty} \|\u^e\|\leq  \| \nabla \J\|_{\mathbb{L}^1} \|\varphi^e\| \|\u^e\| \no \\
		&\leq \frac{m C_0}{8}\|\varphi^e\|^2 + \frac{2}{m C_0}\| \nabla \J\|_{\mathbb{L}^1}^2 \|\u^e\|^2.
		\end{align}
		Similarly, we obtain
		\begin{align}\label{3.18a}
		\left|((\nabla\mathrm{J}*\varphi_1)\varphi^e,\u^e)\right| &\leq \frac{m C_0}{8}\|\varphi^e\|^2 + \frac{2}{m C_0}\| \nabla \J\|_{\mathbb{L}^1}^2 \|\u^e\|^2 .
		\end{align}
		Substituting \eqref{3.15a}, \eqref{3.16a}-\eqref{3.18a} in \eqref{s49}, then using \eqref{gradu estimate} and the fact that $\| \nabla a \|_{\L^\infty} \leq \|\nabla \J \|_{\mathbb{L}^1}$,  for $n =2$, we obtain 
		\begin{subequations}
			\begin{align}\label{3.19a}
			\nu\|\nabla\u^e\|^2  & \leq \frac{3 m C_0}{8}\|\varphi^e\|^2 + \frac{1}{\nu} \left( \frac{2}{\lambda_1} \right)^{\frac{1}{2}} \|\h\|_{\mathbb{V}'_{\mathrm{div}}} \|\nabla\u^e\|^2  +\frac{6}{ m C_0} \|\nabla \J \|_{\mathbb{L}^1}^2   \|\u^e\|^2.
			\end{align}
			Combining \eqref{3.15b}, \eqref{3.16a}-\eqref{3.18a} and substituting it in \eqref{s49}, for $n=3$, we get
			\begin{align} \label{3.19b}
			\nu\|\nabla\u^e\|^2  & \leq \frac{3 m C_0}{8}\|\varphi^e\|^2 + \frac{1}{\nu} \left( \frac{4}{\sqrt{\lambda_1}} \right)^{\frac{1}{2}} \|\h\|_{\mathbb{V}'_{\mathrm{div}}} \|\nabla\u^e\|^2 +\frac{6}{ m C_0} \|\nabla \J \|_{\mathbb{L}^1}^2   \|\u^e\|^2.  
			\end{align}
		\end{subequations}		
		Now we estimate the terms in the right hand side of \eqref{3.14}. To estimate $(\u^e\cdot\nabla\varphi_1,\mathcal{B}^{-1}\varphi^e)$, we use an integration by parts, $\u^e\big|_{\partial\Omega}=0$ and the divergence free condition of $\u^e$ to obtain  
		\begin{align*}
		(\u^e\cdot \nabla \varphi_1,\mathcal{B}^{-1}\varphi^e) = -(\u^e\cdot \nabla \mathcal{B}^{-1}\varphi^e,\varphi_1).
		\end{align*}
		Using H\"older's, Ladyzhenskaya's, Poincar\'e's and Young's inequalities, we estimate the above term as
		\begin{align}\label{3.20a}
		|(\u^e\cdot \nabla \mathcal{B}^{-1}\varphi^e,\varphi_1)| & \leq \|\u^e\| \, \|\nabla \mathcal{B}^{-1}\varphi^e\| \, \|\varphi_1\|_{\mathrm{L}^\infty}\nonumber\\
		&\leq  \|\u^e\| \, \|\nabla \mathcal{B}^{-1}\varphi^e \| \nonumber\\
		&\leq \left( \frac{1}{\sqrt{\lambda_1}} \right) \|\nabla\u^e\|\|\nabla \mathcal{B}^{-1}\varphi^e \|\nonumber\\
		&\leq \frac{\nu}{2} \|\nabla\u^e\|^2 +\frac{1}{2\nu \lambda_1} \|\mathcal{B}^{-1/2}\varphi^e \|^2,
		\end{align}
		where $\lambda_1$ is the first eigenvalue of the Stokes operator $\A$.
		Next we estimate $(\u_2\cdot\nabla\varphi^e,\mathcal{B}^{-1}\varphi^e)$ in the following way:
		\begin{align}\label{3.21}
		|(\u_2\cdot\nabla\varphi^e,\mathcal{B}^{-1}\varphi^e)| &\leq |(\u_2\cdot\nabla \mathcal{B}^{-1}\varphi^e,\varphi^e)| \nonumber \\
		&\leq \|\u_2\|_{\mathbb{L}^4} \|\nabla \mathcal{B}^{-1} \varphi^e\|_{\mathbb{L}^4}\|\varphi^e\| \nonumber \\
		&\leq \frac{m C_0}{8}\|\varphi^e\|^2 + \frac{2}{m C_0}\|\nabla \mathcal{B}^{-1}\varphi^e  \|_{\mathbb{L}^4}^2\|\u_2\|_{\mathbb{L}^4}^2 \nonumber \\
		& \leq \frac{m C_0}{8}\|\varphi^e\|^2 + \frac{2C}{m C_0}  \|\nabla \mathcal{B}^{-1}\varphi^e\|_{\mathbb{H}^1}^2 \|\u_2\|_{\mathbb{L}^4}^2
		\end{align}
		where we used the embedding $\mathbb{H}^1 \hookrightarrow \mathbb{L}^4$.	One can see that the $\mathrm{H}^2$-norm of $\zeta$ in $\D(\mathcal{B})$ is equivalent to the $\mathrm{L}^2$-norm of $(\mathcal{B} + \mathrm{I})\zeta $, that is,
		$$\| \zeta \|_{\mathrm{H}^2} \cong \|(\mathcal{B} + \mathrm{I})\zeta\|.$$
		Now since $  \mathcal{B}^{-1}\varphi^e \in \D(\mathcal{B})$ and $\mathcal{B}$ is a linear isomorphism,  we have 
		\begin{align*}
		\|\nabla \mathcal{B}^{-1}\varphi^e\|_{\mathbb{H}^1} \leq  C\|\mathcal{B}^{-1}\varphi^e \|_{\mathrm{H}^2} \leq C \|(\mathcal{B}+\mathrm{I})\mathcal{B}^{-1}\varphi^e\| \leq C \|\varphi^e\|.
		\end{align*}
		Substituting in \eqref{3.21}, using Ladyzhenskaya's inequality and \eqref{gradu estimate} , for $n=2$,we obtain 
		\begin{subequations}
			\begin{align}\label{3.22}
			|(\u_2\cdot\nabla\varphi^e,\mathcal{B}^{-1}\varphi^e)| 
			&\leq \frac{m C_0}{8}\|\varphi^e\|^2 +  \frac{2C}{m C_0} \|\varphi^e\|^2 \, \|\u_2\|_{\mathbb{L}^4}^2  \nonumber \\
			&\leq \frac{m C_0}{8}\|\varphi^e\|^2 + \frac{2 C }{m C_0}\|\varphi^e\|^2 \left( \frac{\sqrt{2}}{\nu ^2\sqrt{\lambda_1}} \right)  \| \h \|^2 _{\mathbb{V}'_{\mathrm{div}}},
			\end{align}
			and for $n=3$, we get 
			\begin{align} \label{s50}
			|(\u_2\cdot\nabla\varphi^e,\mathcal{B}^{-1}\varphi^e )| \leq \frac{m C_0}{8}\|\varphi^e\|^2 + \frac{2C}{m C_0}\|\varphi^e\|^2 \left( \frac{4}{\sqrt{\lambda_1}} \right)^{\frac{1}{2}} \frac{1}{\nu^2}  \| \h \|^2 _{\mathbb{V}'_{\mathrm{div}}} .
			\end{align}
		\end{subequations}
		It is only left to estimate $(\J\ast \varphi^e , \varphi^e)$. 
		Since $\bar{\varphi}=0$ we can write using the fact that $\|B_N^{1/2} \varphi\|^2 = (B_N \varphi,\varphi) = \|\nabla \varphi\|^2$ for all $\varphi \in D(B_N)$
			\begin{align*}
			|(\J*\varphi,\varphi)|&=|(B_N^{1/2}(\J*\varphi-\overline{\J*\varphi}), B_N^{-1/2} \varphi)| \\
			&\leq \| B_N^{1/2}(\J*\varphi-\overline{\J*\varphi})\| \| B_N^{-1/2} \varphi\| \\
			&\leq \|\nabla \J * \varphi\| \| B_N^{-1/2} \varphi\| \leq \|\nabla \J\|_{\L^1} \|\varphi\| \| B_N^{-1/2} \varphi\| \\
			&\leq \frac{C_0}{4} \|\varphi\|^2 + \frac{1}{C_0}\|\nabla \J \|_{\mathbb{L}^1}^2\| B_N^{-1/2} \varphi\|^2
			\end{align*}
			which implies 
			\begin{align}\label{3.23}
			m|(\J*\varphi,\varphi)|\leq \frac{mC_0}{4} \|\varphi\|^2 + \frac{m}{C_0}\|\nabla \J \|_{\mathbb{L}^1}^2\| B_N^{-1/2} \varphi\|^2
			\end{align}
		Combining \eqref{3.20a}, \eqref{3.22} and \eqref{3.23} and substituting it in \eqref{3.14}, for $n=2$,  we infer 
		\begin{subequations}
			\begin{align} \label{3.24a}
			\frac{5m C_0}{8}\|\varphi^e\|^2 & \leq \frac{\nu}{2}\|\nabla\u^e\|^2 + \left( \frac{1}{2\nu \lambda_1}+\frac{m}{C_0}\|\nabla \J \|_{\mathbb{L}^1}^2 \right)\|\mathcal{B}^{-1/2}\varphi^e \|^2  + \frac{2 C }{\nu ^2m C_0} \left( \frac{2}{\lambda_1} \right)^{\frac{1}{2}}  \| \h \|^2 _{\mathbb{V}'_{\mathrm{div}}}\|\varphi^e \|^2 .
			\end{align}
			Combining \eqref{3.20a}, \eqref{s50} and \eqref{3.23} and substituting in \eqref{3.14}, for $n=3$, we find 
			\begin{align}\label{3.24b}
			\frac{5m C_0}{8}  \|\varphi^e\|^2 &\leq  \frac{\nu}{2}\|\nabla\u^e\|^2 + \left( \frac{1}{2\nu \lambda_1}+\frac{m}{C_0}\|\nabla \J \|_{\mathbb{L}^1}^2 \right) \|\mathcal{B}^{-1/2}\varphi^e\|^2 + \frac{2 C }{\nu^2 m C_0} \left( \frac{4}{\sqrt{\lambda_1}} \right)^{\frac{1}{2}}\| \h \|^2_{\mathbb{V}'_{\mathrm{div}}}  \|\varphi^e\|^2 .
			\end{align}
		\end{subequations}
		Adding \eqref{3.19a} and \eqref{3.24a}, for $n=2$, we obtain 
		\begin{subequations}
			\begin{align}
			&\frac{\nu}{2}  \|\nabla\u^e\|^2 + \frac{m C_0}{4}\|\varphi^e\|^2 \leq \frac{1}{\nu}\left( \frac{2}{\lambda_1} \right)^{\frac{1}{2}}  \|\h\|_{\mathbb{V}'_{\mathrm{div}}} \|\nabla\u^e\|^2 + \frac{6}{ m C_0} \| \nabla \J\|_{\mathbb{L}^1} ^2 \|\u^e\|^2 \no \\
			& \hspace{4cm}+ \left( \frac{1}{2 \nu \lambda_1}+\frac{m}{C_0}\|\nabla \J \|_{\mathbb{L}^1}^2 \right) \|\mathcal{B}^{-1/2}\varphi^e \|^2 + \frac{2 C }{\nu ^2m C_0} \left( \frac{2}{\lambda_1} \right)^{\frac{1}{2}}  \| \h \|^2 _{\mathbb{V}'_{\mathrm{div}}}\|\varphi^e\|^2. 
			\end{align}
			Combining \eqref{3.19b} and \eqref{3.24b}, for $n =3$, we get 
			\begin{align}
			& \frac{\nu}{2} \|\nabla\u^e\|^2 + \frac{m C_0}{4} \|\varphi^e\|^2 \leq \frac{1}{\nu} \left( \frac{4}{\sqrt{\lambda_1}} \right)^{\frac{1}{2}} \|\h\|_{\mathbb{V}'_{\mathrm{div}}}\|\nabla\u^e\|^2 +\frac{6}{m C_0} \| \nabla \J\|_{\mathbb{L}^1}^2 \|\u^e\|^2\no \\
			&\hspace{4cm}+ \left( \frac{1}{2 \nu\lambda_1} + \frac{m}{C_0}\|\nabla \J\|^2_{\mathbb{L}^1}\right) \|\mathcal{B}^{-1/2}\varphi^e \|^2 + \frac{2C }{\nu^2 m C_0} \left( \frac{4}{\sqrt{\lambda_1}} \right)^{\frac{1}{2}}\| \h \|^2_{\mathbb{V}'_{\mathrm{div}}} \|\varphi^e\|^2 .
			\end{align}
		\end{subequations}
		Now using  the continuous embedding of $\mathrm{L}^2_{(0)} \hookrightarrow \mathrm{V}_0'$, that is, $ \|\mathcal{B}^{-1/2}\varphi^e \|^2 \leq C \| \varphi^e \|^2$ and Poincar\'e's  inequality, we further obtain
		\begin{align}\label{3.25}
		&\left[ \lambda_1 \left( \frac{\nu}{2} - \frac{1}{\nu}\left( \frac{2}{\lambda_1} \right)^{\frac{1}{2}}  \|\h\|_{\mathbb{V}'_{\mathrm{div}}} \right) - \frac{6}{ m C_0} \|\nabla \J \|_{\mathbb{L}^1}^2 \right] \| \u^e\|^2 \no \\
		& +\Bigg[\frac{mC_0}{4}  -\frac{mC}{C_0}\|\nabla \J\|^2_{\mathbb{L}^1}- \frac{C}{2 \nu \lambda_1}  
		- \frac{2C }{\nu ^2m C_0} \left( \frac{2}{\lambda_1} \right)^{\frac{1}{2}}  \| \h \|^2 _{\mathbb{V}'_{\mathrm{div}}} \Bigg]\|\varphi^e\|^2\leq 0.
		\end{align}	
		From \eqref{3.25}, uniqueness for $n= 2$ follows provided quantities in both brackets in above inequality are strictly positive. Thus we conclude that for
		\begin{align*}
		& (i) \  \nu^2 >  \frac{2\sqrt{2}}{\sqrt{\lambda_1}} \|\h\|_{\mathbb{V}'_{\mathrm{div}}}  + \frac{12 \nu }{ \lambda_1m C_0} \|\nabla \J \|_{\mathbb{L}^1}^2,\\
		&(ii) \ (\nu m)^2 \left( \frac{C_0}{4} -\frac{C}{C_0} \|\nabla \J\|^2_{\mathbb{L}^1} \right)>  \nu m  \left(\frac{C}{\lambda_1}\right) + \frac{2 C }{ C_0} \left( \frac{2}{\lambda_1} \right)^{\frac{1}{2}}  \| \h \|^2 _{\mathbb{V}'_{\mathrm{div}}} \ \text{ with }\ \|\nabla\mathrm{J}\|^2_{\mathbb{L}^1}<\frac{C^2_0}{4C},
		\end{align*}
		uniqueness follows in two dimensions. Similarly for $n=3$, we obtain
		\begin{align*}
		& \left[ \lambda_1 \Bigg(\frac{\nu}{2 }-\frac{1}{\nu} \left( \frac{4}{\sqrt{\lambda_1}} \right)^{\frac{1}{2}} \|\h\|_{\mathbb{V}'_{\mathrm{div}}}  \Bigg)-\frac{6}{ m C_0} \|\nabla \J \|_{\mathbb{L}^1}^2 \right] \| \u^e\|^2 \no \\
		&+\Bigg[ \frac{mC_0}{4} - \frac{mC}{C_0} \|\nabla \J\|_{\mathbb{L}^1}^2 - \frac{C}{2\nu \lambda_1} - \frac{2 C }{\nu^2 m C_0} \left( \frac{4}{\sqrt{\lambda_1}} \right)^{\frac{1}{2}}\| \h \|^2_{\mathbb{V}'_{\mathrm{div}}} \Bigg] \|\varphi^e \|^2 \leq 0.
		\end{align*}
		Hence, the uniqueness follows provided
		\begin{align*}
		&(i)\ \nu^2  >  \left( \frac{16}{\sqrt{\lambda_1}} \right)^{\frac{1}{2}}  \|\h\|_{\mathbb{V}'_{\mathrm{div}}} +\frac{12 \nu }{ \lambda_1m C_0} \|\nabla \J \|_{\mathbb{L}^1}^2,  \\
		&(ii)\  (\nu m)^2 \left( \frac{C_0}{4} -\frac{C}{C_0} \|\nabla \J\|_{\mathbb{L}^1}^2 \right)>   \nu m  \left(\frac{C}{\lambda_1}\right) + \frac{2 C }{ C_0} \left( \frac{4}{\sqrt{\lambda_1}} \right)^{\frac{1}{2}}  \| \h \|^2 _{\mathbb{V}'_{\mathrm{div}}},  \ \text{ with }\ \|\nabla\mathrm{J}\|^2_{\mathbb{L}^1}<\frac{C^2_0}{4C},
		\end{align*}	
		which completes the proof.
	\end{proof}

		\subsection{Regularity of the weak solution}
		In this subsection, we establish the regularity results for the weak solution to the system \eqref{3.2}. Let $(\mathbf{u}, \mu, \varphi ) \in \V_{\mathrm{div}} \times \mathrm{V} \times (\mathrm{V}\cap \L^2_{(k)}(\Omega))  $ and $\varphi(x) \in (-1,1) \ a.e.$ be the unique weak solution of the system \eqref{3.2}. In the next theorem, we prove  the higher order regularity results for the system \eqref{3.2}. 
		\begin{theorem}\label{regularity}
			Let $\mathbf{h}\in \G_{\mathrm{div}}$ and the assumptions of Theorem \ref{unique-steady} are satisfied. Then  the weak solution $(\mathbf{u}, \mu, \varphi)$ of the system \eqref{3.2} satisfies: 
			\begin{align}\label{s3}
			\mathbf{u}\in\mathbb{H}^2(\Omega),\ \ \mu \in \mathrm{H}^2(\Omega) \  \text{ and } \ \varphi \in \W^{1,p} (\Omega),
			\end{align}
			where $2\leq p < \infty$ if $n=2$ and $2 \leq p \leq 6$ if $n=3.$ 
		\end{theorem}
		\begin{proof}
			\textbf{Step 1:}	Let $ \h \in \G_{\mathrm{div}} $. Then from Theorem \ref{mainthm} and Theorem \ref{unique-steady}, we know that there exists a unique weak solution $ (\u, \varphi, \mu) $ satisfying  
			\begin{align}
			\u \in \V_{\mathrm{div}}, \varphi \in \mathrm{V} \cap \mathrm{L}^2_{(k)} \ \mathrm{and} \ \mu \in \mathrm{V}. \label{reg1}
			\end{align}
			Now, we recall the  form \eqref{nonlin u rewritten} that is
			\begin{align} \label{reg2}
			\nu (\nabla \u, \nabla \v) + b(\u,\u, \v)  = - \left< \nabla a\frac{{\varphi}^2}{2} - (\J\ast \varphi)\nabla \varphi , \v \right>+ \langle \h,\v \rangle, \ \text{ for all }\ \v \in \V_{\mathrm{div}},
			\end{align} 
			where $\nabla a(x) = \int_\Omega \nabla \J(x-y) \d y$. Using the weak regularity \eqref{reg1}, $ |\varphi (x)| \leq 1 \ $ and Young's inequality for convolution, we observe that
			\begin{align*}
			\| (\nabla a) \varphi^2\|& \leq \|\nabla a\|_{\mathbb{L}^{\infty}}\|\varphi\|_{\mathrm{L}^4}^2\leq  C ,\\ 
			\|\J\ast \varphi \nabla \varphi\|& \leq \|\J\|_{\mathrm{L}^1}\|\varphi\|_{\mathrm{L}^\infty}\|\nabla \varphi\| \leq C.
			\end{align*} 
			Hence, the right hand side of \eqref{reg2} belongs to $\mathbb{L}^2(\Omega)$. Then from the regularity of the steady state Navier-Stokes equations (see Chapter II, \cite{MR0609732}) we have that
			\begin{align}
			\u \in \H^2(\Omega). \label{reg4}
			\end{align}
			\vskip 0.1cm \noindent
			\textbf{Step 2:} From the first equation of \eqref{3.2} we have 
			\begin{align}
			m ( \nabla \mu, \nabla \psi) = (\u \cdot \nabla \varphi, \psi),\ \text{ for all }\ \psi \in \mathrm{V}. \label{reg3}
			\end{align}
			Using the Sobolev inequality, \eqref{reg4} and \eqref{reg1}, we observe that
			\begin{align}
			\|\u \cdot \nabla \varphi \| \leq C\|\u\|_{\mathbb{L}^\infty}\|\nabla \varphi\| \leq C \|\u\|_{\H^2} \|\nabla \varphi\| \leq C.
			\end{align}
			Since we have that $\frac{\partial \mu }{\partial \n}=0$, from the $L^p$ regularity of the Neumann Laplacian (see Lemma \ref{Lp_reg}) for $p=2$, we conclude that 
			\begin{align} \label{reg8}
			\mu \in \mathrm{H}^2(\Omega). 
			\end{align}
			\vskip 0.1cm \noindent
			\textbf{Step 3:} An application of the Gagliardo-Nirenberg	inequality together with \eqref{reg8} gives
			\begin{align}\label{s9}
			\|\nabla\mu\|_{\mathbb{L}^p}&\leq C\|\mu\|_{\mathrm{H}^2}^{1-\frac{1}{p}}\|\mu\|^{\frac{1}{p}} < \infty,
			\end{align}	
			for $2\leq p < \infty$ in case of $n=2$ and 
			\begin{align} \label{gradmu_n3}
			\|\nabla\mu\|_{\mathbb{L}^p} &  \leq C\|\mu\|_{\mathrm{H}^2}^{\frac{5p-6}{4p}}\|\mu\|^{\frac{6-p}{4p}} < \infty,
			\end{align}	
			for $2 \leq p \leq 6$ and $n=3.$	
			Let us  take the gradient of $\mu= a\varphi- \mathrm{J}*\varphi + \mathrm{F}'(\varphi)$, multiply it by $\nabla\varphi|\nabla\varphi|^{p-2}$, integrate the resulting identity over $\Omega$,  and use (A9),  H\"older's and Young's inequalities and Young's inequality for convolution to obtain
			\begin{align}\label{s5}
			\int_{\Omega}\nabla\varphi|\nabla\varphi|^{p-2}\cdot \nabla\mu\d x&=\int_{\Omega}\nabla\varphi|\nabla\varphi|^{p-2}\cdot\left(a\nabla\varphi+\nabla a\varphi-\nabla\mathrm{J}*\varphi+\mathrm{F}''(\varphi)\nabla\varphi\right)\d x\nonumber\\
			&=\int_{\Omega}(a+\mathrm{F}''(\varphi))|\nabla\varphi|^p\d x+\int_{\Omega}\nabla\varphi|\nabla\varphi|^{p-2}\cdot\left(\nabla a\varphi-\nabla\mathrm{J}*\varphi\right)\d x\nonumber\\
			&\geq C_0\int_{\Omega}|\nabla\varphi|^p\d x-\left(\|\nabla a\|_{\mathbb{L}^{\infty}}+\|\nabla\mathrm{J}\|_{\mathbb{L}^1}\right) \|\varphi \|_{\mathrm{L}^p} \| \nabla \varphi \|_{\mathbb{L}^p}^{p-1} \nonumber\\
			&\geq \frac{C_0}{2}\|\nabla\varphi\|_{\mathbb{L}^p}^p -\frac{1}{p}\left(\frac{2(p-1)}{C_0p}\right)^{p-1}\left(\|\nabla a\|_{\mathbb{L}^{\infty}}+\|\nabla\mathrm{J}\|_{\mathbb{L}^1}\right)^p\|\varphi\|_{\mathbb{L}^p}^p.
			\end{align}
			Using Young's inequality, we also have 
			\begin{align}\label{s6}
			\left|	\int_{\Omega}\nabla\varphi|\nabla\varphi|^{p-2}\cdot \nabla\mu\d x\right|&\leq \int_{\Omega}|\nabla\varphi|^{p-1}|\nabla\mu|\d x\nonumber\\&\leq \frac{C_0}{4}\|\nabla\varphi\|_{\mathbb{L}^p}^p+\frac{1}{p}\left(\frac{4(p-1)}{C_0p}\right)^{p-1}\|\nabla\mu\|_{\mathbb{L}^p}^p.
			\end{align}
			Combining \eqref{s5} and \eqref{s6}, we get 
			\begin{align}
			\label{s7}
			\frac{C_0}{4}\|\nabla\varphi\|_{\mathbb{L}^p}^p\leq\frac{1}{p}\left(\frac{2(p-1)}{C_0p}\right)^{p-1}\left[2^{p-1}\|\nabla\mu\|_{\mathbb{L}^p}^p+ \left(\|\nabla a\|_{\mathbb{L}^{\infty}}+\|\nabla\mathrm{J}\|_{\mathbb{L}^1}\right)^p\|\varphi\|_{\mathrm{L}^p}^p\right].
			\end{align}
			Using \eqref{s9} and \eqref{gradmu_n3} we get 
			\begin{align*}
			\|\nabla\varphi\|_{\mathbb{L}^p} < \infty
			\end{align*}
			for $2\leq p < \infty$ if $n=2$, and $2 \leq p \leq 6$ if $n=3.$ 
		\end{proof}

	\section{Exponential Stability}\label{se4}\setcounter{equation}{0}
	The stability analysis of non-linear dynamical systems has a long history starting from the works of Lyapunov. For the solutions of ordinary or partial differential equations describing dynamical systems, different kinds of stability may be described. One of the most important types of stability is that concerning the stability of solutions near to a point of equilibrium (stationary solutions).  In the qualitative theory of ordinary and partial differential equations, and control theory, Lyapunov’s notion of (global) asymptotic stability of an equilibrium is a key concept. It is important to note that the asymptotic stability do not quantify the rate of convergence. In fact, there is a strong form of stability which demands an exponential rate of convergence. The notion of exponential stability is far stronger and it assures a minimum rate of decay, that is, an estimate of how  fast the solutions converge to its equilibrium.  In particular, exponential stability implies uniform asymptotic stability. Stability analysis of fluid dynamic models has been one of the essential areas of applied mathematics with a good number of applications in engineering and physics (cf. \cite{MR0609732, MR3186318}, etc). 
	
	In this section, we consider the singular potential $\F$ in $(-1,1)$ of the form: 
	\begin{align*}
	\F (\varphi) = \frac{\theta}{2} ((1+ \varphi) \log (1+\varphi) + (1-\varphi) \log (1-\varphi))  ,\quad \varphi \in (-1,1).
	\end{align*}
	Observe that in this case, the potential $\F$ is convex and $\kappa =0$ in \eqref{decomp of F}. For such potentials, we prove that the stationary solution $(\mathbf{u}^e,\varphi^e)$  of the system (\ref{steadysys}) with constant mobility parameter $m$ and constant coefficient of kinematic viscosity $\nu$ is  exponentially stable in 2-D. That is, our aim is to establish  that:
	\begin{itemize}
		\item 
		there exists constants $M>0$ and $\alpha>0$ such that 
		\begin{align*}
		\|\u(t)-\mathbf{u}^e\|^2+\|\varphi (t)-\varphi^e\|^2\leq M e^{-\alpha t},
		\end{align*} 
		for all $t\geq 0$. 
	\end{itemize}

	\subsection{Global solvability of two dimensional CHNS system} We consider the following initial-boundary value problem:
	\begin{equation}\label{4.1}
	\left\{
	\begin{aligned}
	\varphi_t+ \mathbf{u} \cdot \nabla \varphi &= m\Delta \mu, \ \text{ in }\ \Omega\times(0,T), \\
	\mu &= a \varphi - \mathrm{J}*\varphi + \mathrm{F}'(\varphi) \\
	\mathbf{u}_t -\nu \Delta \mathbf{u} + (\mathbf{u} \cdot \nabla) \mathbf{u }+ \nabla \uppi &= \mu \nabla \varphi +\mathbf{ h}, \ \text{ in }\ \Omega\times(0,T),\\
	\text{div }\mathbf{u}& = 0, \ \text{ in }\ \Omega\times(0,T), \\
	\frac{\partial \mu}{ \partial \mathbf{n}}=0,& \ \mathbf{u}=0, \ \text{ on }\ \partial\Omega\times[0,T], \\
	\mathbf{u}(0) = \mathbf{u}_0,& \quad \varphi(0)=\varphi_0, , \ \text{ in }\ \Omega.
	\end{aligned}
	\right.
	\end{equation}
	Let us now give the global solvability results available in the literature for the system \eqref{4.1}. We first give the definition of a weak solution for the system \eqref{4.1}
	\begin{definition}[weak solution, \cite{MR3019479}]
		Let $\u_0\in\G_{\mathrm{div}}$, $\varphi_0\in \mathrm{H}$ with $\F(\varphi_0)\in\mathrm{L}^1(\Omega)$ and $0<T<\infty$ be given. A couple $(\u,\varphi)$ is a \emph{weak solution} to the system  \eqref{4.1} on $[0,T]$ corresponding to $ [\u_0,\varphi_0]$ if 
		\begin{itemize}
			\item [(i)] $\u,\varphi$ and $\mu$ satisfy 
			\begin{equation}\label{sol}
			\left\{
			\begin{aligned}
			&	\u \in \mathrm{L}^{\infty}(0,T;\G_{\mathrm{div}}) \cap \mathrm{L}^2(0,T;\V_{\mathrm{div}}),  \\ 
			&	\u_t \in \mathrm{L}^{4/3}(0,T;\V_{\mathrm{div}}'),\text{ if } d=3,  \\
			&	\u_t \in \mathrm{L}^{2}(0,T;\V_{\mathrm{div}}'),\text{ if } d=2,  \\
			&	\varphi \in \mathrm{L}^{\infty}(0,T;\mathrm{H}) \cap \mathrm{L}^2(0,T;\mathrm{V}),   \\
			&	\varphi_t \in \mathrm{L}^{2}(0,T;\mathrm{V}'), \\
			& \mu = a \varphi - \J*\varphi + \F'(\varphi) \in L^2(0,T;\mathrm{V}),
			\end{aligned}
			\right.
			\end{equation}
			and 
			$$\varphi \in \L^\infty (Q), \quad |\varphi (x,t)| <1 \ a.e. \ (x,t) \in Q:=\Omega \times (0,T);$$
			\item [(ii)]  For every $\psi\in\mathrm{V}$, every $\v \in \V_{\mathrm{div}}$ and for almost any $t\in(0,T)$, we have
			\begin{align}
			\langle \varphi_t,\psi\rangle + m(\nabla \mu, \nabla \psi) &=\int_{\Omega}(\u\cdot\nabla\psi)\varphi\d x,\\
			\langle \u_t,\v\rangle +\nu(\nabla\u,\nabla\v)+b(\u,\u,\v)&=-\int_{\Omega}(\varphi\cdot\nabla\mu)\v\d x+\langle\h,\v\rangle.
			\end{align}
			\item [(iii)] The initial conditions 
			$\u(0)=\u_0,\ \varphi(0)=\varphi_{0}$ hold in the weak sense.
		\end{itemize}
	\end{definition}
	Next, we discuss the existence and uniqueness of weak solution results available in the literature for  the system  \eqref{4.1}.
	\begin{theorem}[Existence, Theorem 1 \cite{MR3019479}]\label{exist}
		Assume that (A1)-(A7) are satisfied for some fixed positive integer $q$.  Let $\u_0 \in \G_{\mathrm{div}}$, $\varphi_0 \in \mathrm{L}^\infty(\Omega)$ such that $\F(\varphi_0) \in \mathrm{L}^1(\Omega)$ and $\mathbf{h} \in \mathrm{L}^2_{\text{loc}}([0,\infty), \V_{\mathrm{div}}')$. In addition, assume that $|\overline{\varphi_0}|<1$. Then, for every given $T>0$, there exists a weak solution $(\u,\varphi)$ to the equation \eqref{4.1} on $[0,T]$ such that $\overline{\varphi}(t) = \overline{\varphi_0}$ for all $t \in [0,T]$ and 
		\begin{align*}
		\varphi \in \L^\infty (0,T; \L^{2+2q}(\Omega)).
		\end{align*}
		Furthermore, setting
		\begin{align}\mathscr{E}(\u(t),\varphi(t)) = \frac{1}{2} \|\u(t)\|^2 + \frac{1}{4} \int_\Omega \int_\Omega \J(x-y) (\varphi(x,t) - \varphi(y,t))^2 \d x \d y + \int_\Omega \F(\varphi(t))\d x,\end{align}
		the following energy estimate holds 
		\begin{align}\label{energy}
		\mathscr{E}(\u(t),\varphi(t)) + \int_s^t \left(2  \|\sqrt{\nu (\varphi)} D \u(s)\|^2 + m\| \nabla\mu(s) \|^2 \right)\d s \leq \mathscr{E}(\u(s),\varphi(s)) + \int_s^t \langle \mathbf{h}(s), \u(s) \rangle\d s,
		\end{align}
		for all $t \geq s$ and for a.a. $s \in (0,\infty)$.
		If $d=2$, the weak solution $(\u,\varphi)$ satisfies the following energy identity,
		$$\frac{\d}{\d t}\mathscr{E}(\u(t),\varphi(t)) + 2\|\sqrt{\nu (\varphi)} D \u(t) \|^2+ m\| \nabla \mu(t) \|^2 = \langle \mathbf{h}(t) , \u(t) \rangle,$$
		that is, the equality in \eqref{energy} holds for every $t \geq 0.$ 
	\end{theorem} 
	
	\begin{theorem}[Uniqueness, Theorem 3, \cite{MR3518604}]\label{unique}
		Let $d=2$ and assume that (A1)-(A7) are satisfied for some fixed positive integer $q$.  Let $\u_0 \in \G_{\mathrm{div}}$, $\varphi_0 \in \mathrm{L}^\infty (\Omega)$ with $\F(\varphi_0) \in \mathrm{L}^1(\Omega), |\overline{\varphi_0}|<1$ and $\mathbf{h} \in \mathrm{L}^2_{\text{loc}}([0,\infty);\V_{\mathrm{div}}')$. Then, the weak solution $(\u,\varphi)$ corresponding to $(\u_0,\varphi_0)$  given by Theorem \ref{exist} is unique. Furthermore, for $i=1,2$, let $\z_i :=  (\u_i,\varphi_i)$ be two weak solutions corresponding to two initial data $\z_{0i} :=  (\u_{0i},\varphi_{0i})$ and external forces $\h_i$, with $\u_{0i} \in \G_{\mathrm{div}}$, $\varphi_{0i} \in \mathrm{L}^\infty (\Omega)$ such that $\F(\varphi_{0i}) \in \mathrm{L}^1(\Omega), |\overline{\varphi_{0i}}| < \eta ,$ for some constant $\eta \in [0,1)$ and $\h_i \in \mathrm{L}^2_{\text{loc}}([0,\infty);\V_{\mathrm{div}}')$. Then the following continuous dependence estimate holds:
		\begin{align*}
		&\| \u_2(t) - \u_1(t) \|^2 + \| \varphi_2(t) - \varphi_1(t) \|^2_{\mathrm{V}'} \\
		&\quad + \int_0^t \left( \frac{C_0}{2} \| \varphi_2(\tau) - \varphi_1(\tau) \|^2 + \frac{\nu}{4} \|\nabla( \u_2(\tau) - \u_1(\tau)) \|^2 \right) d \tau \\
		&\leq \left( \|\u_2(0) - \u_1(0) \|^2 + \| \varphi_2(0) - \varphi_1(0) \|^2_{\mathrm{V}'} \right) \Lambda_0(t) \\
		&\quad + \| \overline{\varphi}_2(0) - \overline{\varphi}_1(0)\| \mathbb{Q} \left( \mathcal{E}(z_{01}),\mathcal{E}(z_{02}),\| \h_1 \|_{\mathrm{L}^2(0,t;\V_{div}')},\| \h_2 \|_{\mathrm{L}^2(0,t;\V_{div}')}, \eta \right)  \Lambda_1(t) \\
		&\quad + \| \h_2 - \h_1 \|^2_{\mathrm{L}^2(0,T;\V_{\mathrm{div}}')} \Lambda_2(t) ,
		\end{align*} 
		for all $t \in [0,T]$, where $\Lambda_0(t)$, $\Lambda_1(t)$ and $\Lambda_2(t)$ are continuous functions, which depend on the norms of the two solutions. The functions  $\Lambda_i(t)$ also depend on $\F$, $\J$ and $\Omega$, and $\mathbb{Q}$ depends on $\F$, $\J$ ,$\Omega$ and $\eta$.
	\end{theorem}

	\begin{remark}[Remark 3, \cite{MR3019479}]
		The above theorems also imply $\u\in\C([0,T];\G_{\mathrm{div}})$ and $\varphi\in \C([0,T];\mathrm{H})$, for all $T>0$. 
	\end{remark}
	Now, we prove that the stationary solution of \eqref{4.1} is exponentially stable in two dimensions. Let $(\mathbf{u}^e, \varphi^e)$ be the steady-state solution of the system \eqref{3.2}. From Theorem \ref{mainthm}, we know that there exists a unique weak solution for the system \eqref{3.2}.	 Remember that for $\mathbf{h}\in\mathbb{G}_{\mathrm{div}}$,   the weak solution $(\mathbf{u}^e, \mu^e, \varphi^e)$ of the system \eqref{3.2} has the following regularity: 
	\begin{align*}
	\nabla \varphi^e \in\mathrm{L}^p(\Omega),\ \mu^{e} \in \mathrm{H}^2(\Omega)\ \text{ and } \ \mathbf{u}^e\in\mathbb{H}^2(\Omega), \qquad 2 \leq p < \infty.
	\end{align*}
	Now, in the system \eqref{4.1}, we assume that $\mathbf{h}\in\G_{\text{div }}$ is independent of time and the initial data $\u_0\in\G_{\mathrm{div}}$, $\varphi_0\in\mathrm{L}^{\infty}(\Omega)$. Then, there exists a unique weak solution $(\u,\varphi)$ of  \eqref{4.1} such that the system has the regularity given in \eqref{sol}.
	
	\begin{theorem}\label{thmexp} Let $\u_0\in\G_{\mathrm{div}}$, $\varphi_0\in\mathrm{L}^{\infty}(\Omega)$ and $\h\in\G_{\text{div }}$. Then under the Assumptions of Theorem \ref{unique-steady}, for
		\begin{align*}
		&(i)\ \nu^2>\frac{4}{\lambda_1 }\|\nabla\u^e\|^2,\\
		&(ii)\ (C_0-\|\J\|_{\mathrm{L}^1})^2>2 C_{\Omega}^2\left(\frac{1}{2m}\|\mathbf{u}^e\|_{\mathbb{L}^{\infty}}^2+\frac{4}{\nu \sqrt{\lambda_1}}\|\nabla\mu^e\|_{\mathbb{L}^4}^2\right), \ \text{ with }\ C_0>\|\J\|_{\mathrm{L}^1}, \\
		&(iii) \ \overline{\varphi}_0 = \overline{\varphi}^e,
		\end{align*}
		the stationary solution $(\mathbf{u}^e, \varphi^e)$ of \eqref{3.2}, with the regularity given in Theorem \ref{regularity}, is exponentially stable. That is, there exist constants $M>0$ and $\varrho>0$ such that 
		\begin{align*}
		\|\u(t)-\mathbf{u}^e\|^2+\|\varphi(t)-\varphi^e\|^2\leq M e^{-\varrho t},\end{align*} 
		for all $t\geq 0$, where $(\u, \varphi)$ is the unique weak solution of the system \eqref{4.1}, 
		$$M=\frac{\|\mathbf{y}_0\|^2 + 4\|\mathrm{J}\|_{\mathrm{L}^{1}}(\|\varphi_0\|^2+\|\varphi^{e}\|^2)+ \|\F(\varphi_0) \|_{\L^1} + \|\F(\varphi^e)\|_{\L^1} + \|\F'(\varphi^e)\| (\|\varphi_0\|+\|\varphi^{e}\|)}{\min\{(C_0-\|\mathrm{J}\|_{\mathrm{L}^1}),1\}}$$
		and 
		\begin{align}
		\varrho=\min\left\{\left(\lambda_1\nu-\frac{4}{\nu} \|\nabla\u^e\|^2\right) ,\left[\frac{m(C_0-\|\J\|_{\L^1})}{2C_{\Omega}^2}-\frac{1}{(C_0-\|\J\|_{\mathrm{L}^1})}\left(\|\mathbf{u}^e\|_{\mathbb{L}^{\infty}}^2+\frac{4}{\nu \sqrt{\lambda_1}}\|\nabla\mu^e\|_{\mathbb{L}^4}^2\right)\right]\right\}>0.
		\end{align}
		From the hypothesis of Theorems \ref{exist} and \ref{unique}, settings of Theorem \ref{mainthm} (see Lemma \ref{browderproof} also) and Remark \ref{F'phi}, note that the quantity $M$ is finite.
	\end{theorem}
	
	\begin{proof}
		Let us define $\mathbf{y}:=\mathbf{u}-\mathbf{u}^e$, $\psi:=\varphi- \varphi^e$, $ \widetilde{\mu} = \mu - \mu^e $ and $\widetilde{\uppi} :=\uppi - \uppi^e$. Then we know that $(\mathbf{y},\psi )$ satisfies the following system in the weak sense: 
		\begin{eqnarray}\label{4.3}
		\left\{
		\begin{aligned} 
		\psi_t + \mathbf{y} \cdot \nabla \psi + \mathbf{y} \cdot \nabla \varphi^e + \mathbf{u}^e \cdot \nabla \psi =& m\Delta \widetilde{\mu}, \ \text{ in }\ \Omega\times (0,T),\\ 
		\widetilde{\mu}=& a\psi - \mathrm{J}*\psi + \mathrm{F}'(\psi+\varphi^e )- \mathrm{F}'(\varphi^e),\\
		\mathbf{y}_t- \nu \Delta \mathbf{y} + (\mathbf{y }\cdot \nabla ) \mathbf{y} + (\mathbf{y} \cdot \nabla ) \mathbf{u}^e+ (\mathbf{u}^e \cdot \nabla ) \mathbf{y} + \nabla \widetilde{\uppi} =& \widetilde{\mu} \nabla \psi + \widetilde{\mu} \nabla \varphi^e + \mu^e \nabla \psi,\\&\qquad \qquad \qquad \ \text{ in }\ \Omega\times (0,T),   \\
		\text{div }\mathbf{y} = &0, \ \text{ in }\ \Omega\times (0,T),\\
		\frac{\partial \widetilde{\mu}}{ \partial \mathbf{n}}=0, \ \mathbf{y}=&0, \ \text{ on }\ \partial \Omega\times(0,T),\\
		\mathbf{y}(0) = \mathbf{y}_0, \ \psi(0)=&\psi_0, \ \text{ in }\ \Omega.
		\end{aligned}
		\right.
		\end{eqnarray}
		Now consider third equation of \eqref{4.3} and take inner product with $\mathbf{y}(\cdot)$ to obtain 
		\begin{align}
		\label{4.4}
		&\frac{1}{2}\frac{\d}{\d t}\|\mathbf{y}(t)\|^2+ \nu \|\nabla \mathbf{y}(t)\|^2 + b(\mathbf{y}(t),\mathbf{u}^e,\mathbf{y}(t)) \nonumber\\&\quad= (\widetilde{\mu}(t)\nabla\psi(t),\mathbf{y}(t))+ (\widetilde{\mu}(t) \nabla \varphi^e,\mathbf{y}(t)) + (\mu^e \nabla \psi(t),\mathbf{y}(t)),
		\end{align} 
		where we used the fact that $b(\mathbf{y},\mathbf{y},\mathbf{y})=b(\mathbf{u}^e,\mathbf{y},\mathbf{y})=0$ and $(\nabla\widetilde{\uppi},\mathbf{y})=(\widetilde{\uppi},\nabla\cdot\mathbf{y})=0$. 
		An integration by parts yields 
		\begin{align*}
		(\widetilde{\mu}\nabla\psi,\mathbf{y})=-(\psi\nabla\widetilde{\mu},\mathbf{y})-(\psi\widetilde{\mu},\nabla\cdot\mathbf{y})=-(\psi\nabla\widetilde{\mu},\mathbf{y}),
		\end{align*}
		where we used the boundary data and divergence free condition of $\mathbf{y}$. Similarly, we have $(\widetilde{\mu} \nabla \varphi^e,\mathbf{y}) =-(\varphi^e\nabla \widetilde{\mu} ,\mathbf{y})$ and $ (\mu^e \nabla \psi,\mathbf{y})=-( \psi\nabla \mu^e ,\mathbf{y})$. Thus from \eqref{4.4}, we have 
		\begin{align}
		\label{4.5}
		&\frac{1}{2}\frac{\d}{\d t}\|\mathbf{y}(t)\|^2+ \nu \|\nabla \mathbf{y}(t)\|^2 + b(\mathbf{y}(t),\mathbf{u}^e,\mathbf{y}(t)) \nonumber\\&\quad = -(\psi(t)\nabla\widetilde{\mu}(t),\mathbf{y}(t))-(\varphi^e\nabla \widetilde{\mu}(t) ,\mathbf{y}(t))-( \psi(t)\nabla \mu^e ,\mathbf{y}(t)).
		\end{align} 
		Taking inner product of the third equation in \eqref{4.3} with $\widetilde{\mu}(\cdot)$, we obtain 
		\begin{align} \label{4.6}
		(\psi_t(t),\widetilde{\mu}(t) )+ m\|\nabla \widetilde{\mu}(t)\|^2= -( \mathbf{y}(t) \cdot \nabla \psi (t), \widetilde{\mu}(t))- (\mathbf{y}(t) \cdot \nabla \varphi^e, \widetilde{\mu}(t)) - (\mathbf{u}^e \cdot \nabla \psi(t), \widetilde{\mu}(t)). 
		\end{align}
		Using an integration by parts, divergence free condition and boundary value of $\mathbf{y}$, we get 
		\begin{align*}
		( \mathbf{y} \cdot \nabla \psi ,\widetilde{\mu})=-(\psi\nabla\cdot\mathbf{y},\widetilde{\mu})-(\psi\nabla\widetilde{\mu},\mathbf{y})=-(\psi\nabla\widetilde{\mu},\mathbf{y}).
		\end{align*}
		Similarly, we have $(\mathbf{y} \cdot \nabla \varphi^e, \widetilde{\mu}) =-(  \varphi^e \nabla \widetilde{\mu}, \mathbf{y}) $ and $ (\mathbf{u}^e \cdot \nabla \psi, \widetilde{\mu})= -( \psi  \nabla \widetilde{\mu}, \mathbf{u}^e)$. Thus, from \eqref{4.6}, it is immediate that 
		\begin{align}
		\label{4.7}
		(\psi_t(t),\widetilde{\mu} (t))+ m\|\nabla \widetilde{\mu}(t)\|^2= (\psi(t)\nabla\widetilde{\mu}(t),\mathbf{y}(t))+(  \varphi^e \nabla \widetilde{\mu}(t), \mathbf{y}(t)) + ( \psi (t) \nabla \widetilde{\mu}(t), \mathbf{u}^e). 
		\end{align}
		Adding \eqref{4.5} and \eqref{4.7}, we infer that
		\begin{align}\label{4.8}
		&	\frac{1}{2}\frac{\d}{\d t}\|\mathbf{y}(t)\|^2 + \nu \|\nabla \mathbf{y}(t)\|^2 +(\psi_t(t),\widetilde{\mu}(t) )+ m\|\nabla \widetilde{\mu}(t)\|^2\nonumber\\
		&\quad= - b(\mathbf{y}(t),\mathbf{u}^e,\mathbf{y}(t))-( \psi(t)\nabla \mu^e ,\mathbf{y}(t))+ ( \psi(t)  \nabla \widetilde{\mu}(t), \mathbf{u}^e). 
		\end{align}
		We estimate the term $(\psi_t,\widetilde{\mu} )$  from \eqref{4.8} as 
		\begin{align}\label{4.10}
		(\psi_t, \widetilde{\mu})&= (\psi_t,a\psi - \mathrm{J}*\psi + \mathrm{F}'(\psi+\varphi^e )- \mathrm{F}'(\varphi^e) ) \\ \nonumber
		&=\frac{\d}{\d t} \left\{ \frac{1}{2} \|\sqrt{a} \psi \|^2 -\frac{1}{2} (\mathrm{J}*\psi,\psi) + \int_\Omega \mathrm{F}(\psi+\varphi^e) \d x- (\mathrm{F}'(\varphi^e),\psi)\right\} \nonumber\\
		&= \frac{\d}{\d t} \left\{ \frac{1}{2} \|\sqrt{a} \psi \|^2 -\frac{1}{2} (\mathrm{J}*\psi,\psi) + \int_\Omega \mathrm{F}(\psi+\varphi^e) \d x-\int_\Omega \mathrm{F}(\varphi^e) \d x- \int_{\Omega}\mathrm{F}'(\varphi^e)\psi\d x\right\}, \label{6.16}
		\end{align}
		since $\frac{\d}{\d t}\left(\int_\Omega \mathrm{F}(\varphi^e) \d x\right)=0$. Using H\"older's, Ladyzhenskaya's and Young's inequalities, we estimate $b(\mathbf{y},\mathbf{u}^e,\mathbf{y})$ as 
		\begin{align}\label{4.13}
		|b(\mathbf{y},\mathbf{u}^e,\mathbf{y})|\leq \|\nabla\mathbf{u}^e\|\|\mathbf{y}\|_{\mathbb{L}^4}^2\leq \sqrt{2}\|\nabla\mathbf{u}^e\|\|\mathbf{y}\|\|\nabla\y\|\leq \frac{\nu}{4}\|\nabla\y\|^2+\frac{2}{\nu}\|\nabla\u^e\|^2\|\y\|^2.
		\end{align}
		We estimate the term $( \psi\nabla \mu^e ,\mathbf{y})$ from \eqref{4.8} using H\"older's, Ladyzhenskaya's and Young's inequalities as 
		\begin{align}\label{4.14}
		|( \psi\nabla \mu^e ,\mathbf{y})|&\leq \|\psi\|\|\nabla\mu^e\|_{\mathbb{L}^4}\|\y\|_{\mathbb{L}^4}\leq \sqrt{2}\|\psi\|\|\nabla\mu^e\|_{\mathbb{L}^4}\|\y\|^{1/2}\|\nabla\y\|^{1/2}\nonumber\\
		&\leq \sqrt{2} (\lambda_1)^{1/4}\|\psi\|   \|\nabla\mu^e\|_{\mathbb{L}^4} \|\nabla\y\|\nonumber\\
		&\leq \frac{\nu}{4}\|\nabla\y\|^2+ \frac{2}{\nu \sqrt{\lambda_1}}\|\nabla\mu^e\|_{\mathbb{L}^4}^2 \|\psi\|^2.
		\end{align}
		Similarly, we estimate the term $ ( \psi  \nabla \widetilde{\mu}, \mathbf{u}^e)$ from \eqref{4.8} as
		\begin{align}\label{4.15}
		|( \psi  \nabla \widetilde{\mu}, \mathbf{u}^e)|&\leq\|\psi\|\|\nabla \widetilde{\mu}\|\|\mathbf{u}^e\|_{\mathbb{L}^{\infty}}\leq \frac{m}{2}\|\nabla \widetilde{\mu}\|^2+\frac{1}{2m}\|\mathbf{u}^e\|_{\mathbb{L}^{\infty}}^2\|\psi\|^2.
		\end{align}
		Combining \eqref{6.16}-\eqref{4.15} and substituting it in \eqref{4.8}, we obtain 
		\begin{align}
		\label{4.16}
		&	\frac{\d}{\d t}\left\{\frac{1}{2}\|\mathbf{y}\|^2 + \frac{1}{2} \|\sqrt{a} \psi \|^2 -\frac{1}{2} (\mathrm{J}*\psi,\psi) + \int_\Omega \mathrm{F}(\psi+\varphi^e) \d x-\int_\Omega \mathrm{F}(\varphi^e) \d x- \int_{\Omega}\mathrm{F}'(\varphi^e)\psi\d x\right\} \no \\
		&+ \frac{\nu}{2} \|\nabla \mathbf{y}\|^2 + \frac{m}{2}\|\nabla \widetilde{\mu}\|^2 	\leq \left(\frac{1}{2m}\|\mathbf{u}^e\|_{\mathbb{L}^{\infty}}^2+\frac{2}{\nu \sqrt{\lambda_1}}\|\nabla\mu^e\|_{\mathbb{L}^4}^2\right)\|\psi\|^2+\frac{2}{\nu}\|\nabla\u^e\|^2\|\y\|^2.
		\end{align}
		Using Taylor's series expansion and Assumption (A9),  we also obtain 
		\begin{align}\label{433}
		(\widetilde{\mu},\psi)&=(a\psi - \mathrm{J}*\psi + \mathrm{F}'(\psi+\varphi^e )- \mathrm{F}'(\varphi^e), \psi)\nonumber\\
		&=(a\psi+\mathrm{F}''(\varphi^e +\theta\psi)\psi,\psi)-(\mathrm{J}*\psi,\psi)\nonumber\\
		&\geq (C_0-\|\mathrm{J}\|_{\mathrm{L}^1})\|\psi\|^2,
		\end{align}
		for some $0<\theta<1$.  Observe  that $(\overline{\tilde{\mu}}, \psi)=\overline{\tilde{\mu}}(1,\psi)=0$, since $\overline{\varphi}_0=\overline{\varphi}= \overline{\varphi}^e$. Thus  an application of Poincar\'e inequality yields 
		\begin{align}\label{exp1}
		(\tilde{\mu}, \psi)=(\tilde{\mu}-\overline{\tilde{\mu}}, \psi)\leq\|\tilde{\mu}-\overline{\tilde{\mu}}\|\|\psi\|\leq C_\Omega\|\psi\| \|\nabla \tilde{\mu}\|.
		\end{align}
		Hence, from \eqref{433} and \eqref{exp1}, we get 
		\begin{align} \label{exp3}
			\|\psi\| \leq \frac{C_\Omega}{C_0 - \|\J\|_{\L^1}} \|\nabla \tilde{\mu}\|,
		\end{align}
		which implies from \eqref{exp1} that
		\begin{align}
			(\tilde{\mu}, \psi) \leq \frac{C_\Omega^2}{C_0 - \|\J\|_{\L^1}}\|\nabla \tilde{\mu}\|^2 \label{exp2}
		\end{align}
		Using Poincar\'e's inequality and \eqref{exp3} in \eqref{4.16}, we get
		\begin{align}
		&	\frac{\d}{\d t}\left\{\frac{1}{2}\|\mathbf{y}\|^2 + \frac{1}{2} \|\sqrt{a} \psi \|^2 -\frac{1}{2} (\mathrm{J}*\psi,\psi) + \int_\Omega \mathrm{F}(\psi+\varphi^e) \d x-\int_\Omega \mathrm{F}(\varphi^e) \d x- \int_{\Omega}\mathrm{F}'(\varphi^e)\psi\d x\right\} \no \\
		&+ \left( \frac{\nu \lambda_1}{2} -\frac{2}{\nu}\|\nabla\u^e\|^2 \right)\|\y\|^2+ \left( \frac{m}{2} -\left( \frac{1}{2m}\|\mathbf{u}^e\|_{\mathbb{L}^{\infty}}^2-\frac{2}{\nu \sqrt{\lambda_1}}\|\nabla\mu^e\|_{\mathbb{L}^4}^2 \right) \frac{C_\Omega^2}{(C_0 - \|\J\|_{\L^1})^2}\right) \|\nabla \widetilde{\mu}\|^2 	\leq 0.
		\end{align}
		Using Taylor's formula, we have the following identities:
		\begin{align} \label{exp4}
		&\frac{1}{2} \|\sqrt{a} \psi \|^2 -\frac{1}{2} (\mathrm{J}*\psi,\psi)+\int_\Omega \left[\mathrm{F}(\psi+\varphi^e)-\mathrm{F}(\varphi^e)- \mathrm{F}'(\varphi^e)\psi\right]\d x \no \\
		& =(a\psi - \mathrm{J}*\psi, \psi)+ \frac{1}{2}\int_{\Omega}\mathrm{F}''(\varphi^{e}+\theta\psi)\psi^2\d x, \no \\
		&=(a\psi - \mathrm{J}*\psi + \mathrm{F}'(\psi+\varphi^e )- \mathrm{F}'(\varphi^e), \psi) = (\tilde{\mu}, \psi)
		\end{align}
		for some $0<\theta<1$. 
		Using (A9) and Young's inequality for convolutions, we know that 
		\begin{align*}
		\int_{\Omega}\left(a+\mathrm{F}''(\varphi^{e}+\theta\psi)\right)\psi^2\d x\geq C_0\|\psi\|^2\geq \|\J\|_{\mathbb{L}^1}\|\psi\|^2 
		\geq (\J*\psi,\psi). 
		\end{align*}
		Thus, we have 
		\begin{align*}
		\int_{\Omega}\left(a+\mathrm{F}''(\varphi^{e}+\theta\psi)\right)\psi^2\d x- (\mathrm{J}*\psi,\psi)\geq 0.
		\end{align*}
		Hence from \eqref{exp2}, we obtain 
		\begin{align}\label{436}
		&	\frac{\d}{\d t}\left[\|\mathbf{y}\|^2 + (\widetilde{\mu},\psi)\right] + \left(\nu \lambda_1-\frac{4}{\nu}\|\nabla\u^e\|^2\right) \|\mathbf{y}\|^2 \nonumber\\&
		\quad + \left[\frac{(C_0-\|\J\|_{\mathrm{L}^1})}{2C_{\Omega}^2}-\frac{1}{(C_0-\|\J\|_{\mathrm{L}^1})}\left(\frac{1}{2m}\|\mathbf{u}^e\|_{\mathbb{L}^{\infty}}^2+\frac{4}{\nu \sqrt{\lambda_1}}\|\nabla\mu^e\|_{\mathbb{L}^4}^2\right)\right] (\widetilde{\mu},\psi)\leq 0.
		\end{align}
		Now for $C_0>\|\J\|_{\mathrm{L}^1}$ and 
		\begin{align}
		\nu^2&>\frac{4}{\lambda_1 }\|\nabla\u^e\|^2, \ \text {and }\ 
		(C_0-\|\J\|_{\mathrm{L}^1})^2>2 C_{\Omega}^2\left(\frac{1}{2m}\|\mathbf{u}^e\|_{\mathbb{L}^{\infty}}^2+\frac{4}{\nu \sqrt{\lambda_1}}\|\nabla\mu^e\|_{\mathbb{L}^4}^2\right),
		\end{align}
		we use the variation of constants formula (see Lemma \ref{C3}) to find 
		\begin{align} \label{s51}
		\|\mathbf{y}(t)\|^2 + (\widetilde{\mu}(t),\psi(t))\leq \left(\|\mathbf{y}(0)\|^2 + (\widetilde{\mu}(0),\psi(0))\right)e^{-\varrho t},
		\end{align}
		for all $t\geq 0$, where 
		\begin{align}
		\varrho=\min\left\{\left(\lambda_1\nu-\frac{4}{\nu} \|\nabla\u^e\|^2\right) ,\left[\frac{(C_0-\|\J\|_{\mathrm{L}^1})}{2C_{\Omega}^2}-\frac{1}{(C_0-\|\J\|_{\mathrm{L}^1})}\left(\frac{1}{2m}\|\mathbf{u}^e\|_{\mathbb{L}^{\infty}}^2+\frac{4}{\nu \sqrt{\lambda_1}}\|\nabla\mu^e\|_{\mathbb{L}^4}^2\right)\right]\right\}>0.
		\end{align}
		From \eqref{433}, we also have $ (C_0-\|\mathrm{J}\|_{\mathrm{L}^1})\|\psi\|^2\leq (\widetilde{\mu},\psi),$ so that from \eqref{s51}, we infer that 
		\begin{align}\label{442}
		\|\mathbf{y}(t)\|^2 + (C_0-\|\mathrm{J}\|_{\mathrm{L}^1})\|\psi\|^2\leq \left(\|\mathbf{y}(0)\|^2 + (\widetilde{\mu}(0),\psi(0))\right)e^{-\varrho t}.
		\end{align}
		Using H\"older's inequality, Young's inequality for convolutions and assumption (A1) in the identity \eqref{exp4}, we obtain 
			\begin{align}
			(\widetilde{\mu}(0),\psi(0)) &= (a \psi_0, \psi_0)-(\J*\psi_0, \psi_0) + \int_\Omega \F(\varphi_0) \d x+ \int_\Omega \F(\varphi^e)\d x - \int_\Omega \F'(\varphi^e) \psi_0\d x\no \\
			&\leq 2\|\mathrm{J}\|_{\mathrm{L}^{1}}\|\psi_0\|^2+ \|\F(\varphi_0) \|_{\L^1} + \|\F(\varphi^e)\|_{\L^1} + \|\F'(\varphi^e)\|_{\L^1} \|\psi_0\|_{\L^{\infty}} ,
			\end{align}
		where we have used the boundedness of $\psi_0$ and $\|a\|_{\mathrm{L}^{\infty}} \leq \|\mathrm{J}\|_{\mathrm{L}^{1}}$.	Hence from \eqref{442}, we finally have 
		\begin{align}\label{444}
		\|\mathbf{y}(t)\|^2 + \|\psi(t)\|^2  \leq \left( \frac{\|\mathbf{y}_0\|^2 + 2\|\mathrm{J}\|_{\mathrm{L}^{1}}\|\psi_0\|^2+ \|\F(\varphi_0) \|_{\L^1} + \|\F(\varphi^e)\|_{\L^1} + \|\F'(\varphi^e)\| \|\psi_0\| }{\min\{(C_0-\|\mathrm{J}\|_{\mathrm{L}^1}),1\}} \right) e^{-\varrho t},
		\end{align}
		which completes the proof. 	
	\end{proof}	

	\appendix 
	\section{Variation of Constants Formula}
	\renewcommand{\thesection}{\Alph{section}}
	\numberwithin{equation}{section}\setcounter{equation}{0}
	In this appendix, we give a variant of variation of constants formula, which is useful when we have two or more differentiable functions with different constant coefficients. 
	\begin{lemma}[Lemma C.3, \cite{MR4120851}]\label{C3}
		Assume that the differentiable functions $y(\cdot),z(\cdot):[0,T]\to[0,\infty)$ and the constants $a_1,a_2,k_1,k_2,k_3>0$ satisfy:
		\begin{align}\label{c1}
		\frac{\d }{\d t}(a_1 y(t)+a_2z(t))+k_1y(t)+k_2z(t)\leq 0,
		\end{align}
		for all $t\in[0,T]$. Then, we have 
		\begin{align}\label{c2}
		y(t)+z(t)\leq C(y(0)+z(0))e^{-\uprho t},\ \text{ where }\ C=\frac{\max\{a_1,a_2\}}{\min\{a_1,a_2\}}\ \text{ and }\ \uprho=\min\left\{\frac{k_1}{a_1},\frac{k_2}{a_2}\right\}.
		\end{align}
	\end{lemma}
	\begin{proof}
		Since $a_1>0$, from \eqref{c1}, we have 
		\begin{align*}
		\frac{\d }{\d t}\left(y(t)+\frac{a_2}{a_1}z(t)\right)+\frac{k_1}{a_1}\left(y(t)+\frac{k_2}{k_1}z(t)\right)\leq 0.
		\end{align*}
		Now, for $\frac{a_2}{a_1}\leq \frac{k_2}{k_1}$, from the above inequality, we also have 
		\begin{align*}
		\frac{\d }{\d t}\left(y(t)+\frac{a_2}{a_1}z(t)\right)+\frac{k_1}{a_1}\left(y(t)+\frac{a_2}{a_1}z(t)\right)\leq 0.
		\end{align*}
		From the above relation, it is immediate that 
		\begin{align*}
		\frac{\d}{\d t}\left[e^{\frac{k_1}{a_1}t}\left(y(t)+\frac{a_2}{a_1}z(t)\right)\right]\leq 0,
		\end{align*}
		which easily implies 
		\begin{align}\label{c3}
		a_1y(t)+a_2z(t)\leq (a_1y(0)+a_2z(0))e^{-\frac{k_1}{a_1}t}.
		\end{align}
		We can do a similar calculation by a division with $a_2>0$ and for $\frac{a_1}{a_2}\leq \frac{k_1}{k_2}$,   we arrive at 
		\begin{align}\label{c4}
		a_1y(t)+a_2z(t)\leq (a_1y(0)+a_2z(0))e^{-\frac{k_2}{a_2}t}.
		\end{align}
		Combining \eqref{c3} and \eqref{c4}, we finally obtain \eqref{c2}. 
	\end{proof}

	\subsection*{Acknowledgements} 
	The authors  thank the anonymous referee for her/his careful reading of the manuscript and for the valuable comments and suggestions, which helped us to improve the manuscript significantly. The authors would like to thank Dr. Dhanya Rajendran from IISER Thiruvananthapuram, INDIA, for valuable discussions and suggestions by her. M. T. Mohan would  like to thank the Department of Science and Technology (DST), Govt of India for Innovation in Science Pursuit for Inspired Research (INSPIRE) Faculty Award (IFA17-MA110).
	
	\bibliographystyle{plain}
	\bibliography{ref_stationary} 

\begin{thebibliography}{10}

\bibitem{MR2563636}
H.~Abels.
\newblock On a diffuse interface model for two-phase flows of viscous,
  incompressible fluids with matched densities.
\newblock {\em Arch. Ration. Mech. Anal.}, 194(2):463--506, 2009.

\bibitem{MR3524178}
H.~Abels and J.~Weber.
\newblock Stationary solutions for a {N}avier-{S}tokes/{C}ahn-{H}illiard system
  with singular free energies.
\newblock In {\em Recent developments of mathematical fluid mechanics}, Adv.
  Math. Fluid Mech., pages 25--41. Birkh\"{a}user/Springer, Basel, 2016.

\bibitem{MR2347608}
H.~Abels and M.~Wilke.
\newblock Convergence to equilibrium for the {C}ahn-{H}illiard equation with a
  logarithmic free energy.
\newblock {\em Nonlinear Anal.}, 67(11):3176--3193, 2007.

\bibitem{MR1195128}
V.~Barbu.
\newblock {\em Analysis and control of nonlinear infinite-dimensional systems},
  volume 190 of {\em Mathematics in Science and Engineering}.
\newblock Academic Press, Inc., Boston, MA, 1993.

\bibitem{MR3186318}
V.~Barbu.
\newblock {\em Stabilization of {N}avier-{S}tokes flows}.
\newblock Communications and Control Engineering Series. Springer, London,
  2011.

\bibitem{MR4108622}
T.~Biswas, S.~Dharmatti, and M.~T. Mohan.
\newblock Maximum principle for some optimal control problems governed by 2{D}
  nonlocal {C}ahn-{H}illard-{N}avier-{S}tokes equations.
\newblock {\em J. Math. Fluid Mech.}, 22(3):Art. 34, 42, 2020.

\bibitem{MR4131779}
T.~Biswas, S.~Dharmatti, and M.~T. Mohan.
\newblock Pontryagin maximum principle and second order optimality conditions
  for optimal control problems governed by 2d nonlocal
  cahn-hilliard-navierstokes equations.
\newblock {\em Analysis (Berlin)}, 40(3):127--150, 2020.

\bibitem{MR1700669}
F.~Boyer.
\newblock Mathematical study of multi-phase flow under shear through order
  parameter formulation.
\newblock {\em Asymptot. Anal.}, 20(2):175--212, 1999.

\bibitem{MR2759829}
H.~Brezis.
\newblock {\em Functional Analysis, Sobolev Spaces and Partial Differential
  Equations}.
\newblock Springer-Verlag, New York, 2010.

\bibitem{MR2834896}
P.~Colli, S.~Frigeri, and M.~Grasselli.
\newblock Global existence of weak solutions to a nonlocal
  {C}ahn-{H}illiard-{N}avier-{S}tokes system.
\newblock {\em J. Math. Anal. Appl.}, 386(1):428--444, 2012.

\bibitem{MR1230384}
E.~DiBenedetto.
\newblock {\em Degenerate parabolic equations}.
\newblock Universitext. Springer-Verlag, New York, 1993.

\bibitem{MR2597943}
L.~C. Evans.
\newblock {\em Partial differential equations}, volume~19 of {\em Graduate
  Studies in Mathematics}.
\newblock American Mathematical Society, Providence, RI, second edition, 2010.

\bibitem{frigeri10regularity}
S.~Frigeri, C.~G. Gal, and M.~Grasselli.
\newblock Regularity results for the nonlocal cahn-hilliard equation with
  singular potential and degenerate mobility.

\bibitem{MR3518604}
S.~Frigeri, C.~G. Gal, and M.~Grasselli.
\newblock On nonlocal {C}ahn-{H}illiard-{N}avier-{S}tokes systems in two
  dimensions.
\newblock {\em J. Nonlinear Sci.}, 26(4):847--893, 2016.

\bibitem{MR3903266}
S.~Frigeri, C.~G. Gal, M.~Grasselli, and J.~Sprekels.
\newblock Two-dimensional nonlocal {C}ahn-{H}illiard-{N}avier-{S}tokes systems
  with variable viscosity, degenerate mobility and singular potential.
\newblock {\em Nonlinearity}, 32(2):678--727, 2019.

\bibitem{MR3000606}
S.~Frigeri and M.~Grasselli.
\newblock Global and trajectory attractors for a nonlocal
  {C}ahn-{H}illiard-{N}avier-{S}tokes system.
\newblock {\em J. Dynam. Differential Equations}, 24(4):827--856, 2012.

\bibitem{MR3019479}
S.~Frigeri and M.~Grasselli.
\newblock Nonlocal {C}ahn-{H}illiard-{N}avier-{S}tokes systems with singular
  potentials.
\newblock {\em Dyn. Partial Differ. Equ.}, 9(4):273--304, 2012.

\bibitem{MR3090070}
S.~Frigeri, M.~Grasselli, and P.~Krej$\check{c}$\'{\i}.
\newblock Strong solutions for two-dimensional nonlocal
  {C}ahn-{H}illiard-{N}avier-{S}tokes systems.
\newblock {\em J. Differential Equations}, 255(9):2587--2614, 2013.

\bibitem{MR4104524}
S.~Frigeri, M.~Grasselli, and J.~Sprekels.
\newblock Optimal distributed control of two-dimensional nonlocal
  {C}ahn-{H}illiard-{N}avier-{S}tokes systems with degenerate mobility and
  singular potential.
\newblock {\em Appl. Math. Optim.}, 81(3):899--931, 2020.

\bibitem{MR3456388}
S.~Frigeri, E.~Rocca, and J.~Sprekels.
\newblock Optimal distributed control of a nonlocal
  {C}ahn-{H}illiard/{N}avier-{S}tokes system in two dimensions.
\newblock {\em SIAM J. Control Optim.}, 54(1):221--250, 2016.

\bibitem{MR3688414}
C.~G. Gal, A.~Giorgini, and M.~Grasselli.
\newblock The nonlocal {C}ahn-{H}illiard equation with singular potential:
  well-posedness, regularity and strict separation property.
\newblock {\em J. Differential Equations}, 263(9):5253--5297, 2017.

\bibitem{MR2580516}
C.~G. Gal and M.~Grasselli.
\newblock Asymptotic behavior of a {C}ahn-{H}illiard-{N}avier-{S}tokes system
  in 2{D}.
\newblock {\em Ann. Inst. H. Poincar\'{e} Anal. Non Lin\'{e}aire},
  27(1):401--436, 2010.

\bibitem{MR1404829}
M.~E. Gurtin, D.~Polignone, and J.~Vi\~{n}als.
\newblock Two-phase binary fluids and immiscible fluids described by an order
  parameter.
\newblock {\em Math. Models Methods Appl. Sci.}, 6(6):815--831, 1996.

\bibitem{MR2108884}
J.~Han.
\newblock The {C}auchy problem and steady state solutions for a nonlocal
  {C}ahn-{H}illiard equation.
\newblock {\em Electron. J. Differential Equations}, pages No. 113, 9, 2004.

\bibitem{MR0254401}
O.~A. Ladyzhenskaya.
\newblock {\em The mathematical theory of viscous incompressible flow}.
\newblock Second English edition, revised and enlarged. Translated from the
  Russian by Richard A. Silverman and John Chu. Mathematics and its
  Applications, Vol. 2. Gordon and Breach, Science Publishers, New
  York-London-Paris, 1969.

\bibitem{MR3726909}
G.~Leoni.
\newblock {\em A first course in {S}obolev spaces}, volume 181 of {\em Graduate
  Studies in Mathematics}.
\newblock American Mathematical Society, Providence, RI, second edition, 2017.

\bibitem{SJL_1961-1962____A6_0}
J.~L. Lions.
\newblock Quelques remarques sur les probl\`emes de dirichlet et de neumann.
\newblock {\em S\'eminaire Jean Leray}, 1961-1962.
\newblock talk:6.

\bibitem{MR4120851}
M.~T. Mohan.
\newblock On the two-dimensional tidal dynamics system: stationary solution and
  stability.
\newblock {\em Appl. Anal.}, 99(10):1795--1826, 2020.

\bibitem{MR109940}
L.~Nirenberg.
\newblock On elliptic partial differential equations.
\newblock {\em Ann. Scuola Norm. Sup. Pisa Cl. Sci. (3)}, 13:115--162, 1959.

\bibitem{MR262827}
R.~T. Rockafellar.
\newblock On the maximal monotonicity of subdifferential mappings.
\newblock {\em Pacific J. Math.}, 33:209--216, 1970.

\bibitem{MR3014456}
T.~Roub\'{\i}$\check{c}$ek.
\newblock {\em Nonlinear partial differential equations with applications},
  volume 153 of {\em International Series of Numerical Mathematics}.
\newblock Birkh\"{a}user/Springer Basel AG, Basel, second edition, 2013.

\bibitem{ruzicka2006nichtlineare}
M.~Ruzicka.
\newblock {\em Nichtlineare Funktionalanalysis: Eine Einf{\"u}hrung}.
\newblock Springer-Verlag, 2006.

\bibitem{MR3436705}
T.~Tachim~Medjo.
\newblock Optimal control of a {C}ahn-{H}illiard-{N}avier-{S}tokes model with
  state constraints.
\newblock {\em J. Convex Anal.}, 22(4):1135--1172, 2015.

\bibitem{MR3555135}
T.~Tachim~Medjo.
\newblock A {C}ahn-{H}illiard-{N}avier-{S}tokes model with delays.
\newblock {\em Discrete Contin. Dyn. Syst. Ser. B}, 21(8):2663--2685, 2016.

\bibitem{MR3565933}
T.~Tachim~Medjo.
\newblock Robust control of a {C}ahn-{H}illiard-{N}avier-{S}tokes model.
\newblock {\em Commun. Pure Appl. Anal.}, 15(6):2075--2101, 2016.

\bibitem{MR0609732}
R.~Temam.
\newblock {\em Navier-{S}tokes equations. {T}heory and numerical analysis}.
\newblock North-Holland Publishing Co., Amsterdam-New York-Oxford, 1977.
\newblock Studies in Mathematics and its Applications, Vol. 2.

\bibitem{MR1318914}
R.~Temam.
\newblock {\em Navier-{S}tokes equations and nonlinear functional analysis},
  volume~66 of {\em CBMS-NSF Regional Conference Series in Applied
  Mathematics}.
\newblock Society for Industrial and Applied Mathematics (SIAM), Philadelphia,
  PA, second edition, 1995.

\bibitem{MR1033498}
E.~Zeidler.
\newblock {\em Nonlinear functional analysis and its applications. {II}/{B}}.
\newblock Springer-Verlag, New York, 1990.
\newblock Nonlinear monotone operators, Translated from the German by the
  author and Leo F. Boron.

\end{thebibliography}
	
\end{document}